\documentclass[reqno,11pt]{amsart}

\usepackage{geometry}
\usepackage{graphicx}
\usepackage{epstopdf}

\usepackage{amsmath,amssymb}
\usepackage{bm}
\usepackage{mathtools}
\usepackage{stmaryrd}
\usepackage{mathrsfs}
\usepackage{euscript}
\usepackage{amsfonts}

\usepackage{amsthm}
\usepackage{empheq}
\usepackage{cases}
\numberwithin{equation}{section}

\usepackage{caption}
\usepackage{subcaption}
\usepackage{booktabs}
\usepackage{array}
\usepackage{tabularx}
\usepackage{multirow}
\usepackage{xcolor}

\usepackage[numbers,sort]{natbib}
\usepackage[colorlinks=true, linkcolor=blue, citecolor=blue, urlcolor=cyan]{hyperref}

\usepackage{enumitem}
\usepackage{float}

\newtheorem{theorem}{Theorem}[section]
\newtheorem{lemma}[theorem]{Lemma}
\newtheorem{proposition}[theorem]{Proposition}
\newtheorem{corollary}[theorem]{Corollary}
\theoremstyle{definition}

\newtheorem{example}[theorem]{Example}

\theoremstyle{remark}
\newtheorem{remark}{Remark}

\newtheoremstyle{assumpstyle}
{8pt}
{8pt}
{\normalfont}
{}
{\bfseries}
{.}
{ }
{}

\theoremstyle{assumpstyle}
\newtheorem{assumption}{Assumption}

\title[Gradient flow for action ground state of RNLS]{Analysis of gradient flow for computing defocusing action ground states of rotating nonlinear Schr\"odinger equations}

\author[W. Liu]{Wei Liu}
\address{W.~Liu: College of Science, National University of Defense Technology, Changsha, 410073, China}
\email{wl@nudt.edu.cn}

\author[T. Wang]{Tingfeng Wang}
\address{T.~Wang: School of Mathematics and Statistics, Wuhan University, Wuhan, 430072, China}
\email{tingfengwang@whu.edu.cn}

\author[Y. Yuan]{Yongjun Yuan}
\address{Y.~Yuan: MOE-LCSM, School of Mathematics and Statistics, Hunan Normal University, Changsha, 410081, China}
\email{yyj1983@hunnu.edu.cn}

\author[X. Zhao]{Xiaofei Zhao}
\address{X.~Zhao: School of Mathematics and Statistics \& Computational Sciences Hubei Key Laboratory, Wuhan University, Wuhan, 430072, China}
\email{matzhxf@whu.edu.cn}

\date{}

\begin{document}

\maketitle

\begin{abstract}
This work focuses on the numerical computation of defocusing action ground states for rotating nonlinear Schr\"odinger equations (RNLS) using a direct gradient flow (DGF) method. We address theoretical gaps in the existing literature concerning the stability and convergence of this DGF scheme. Firstly, we prove the unconditional stability of the DGF scheme, demonstrating that the action functional is monotonically non-increasing along the discrete flow for arbitrary time step sizes. Secondly, we establish a rigorous convergence analysis, proving global convergence under minor assumptions and local exponential convergence to the action ground state under a reasonable non-degeneracy condition. The analysis relies on the uniform boundedness of sublevel sets of the action functional and introduces a tailored $H^1$-distance between phase-shift equivalence classes to handle complex-valued ground states with quantized vortices. A novel analytical framework is also developed to establish the exponential convergence rate. Numerical experiments are presented to validate the theoretical findings, demonstrating both the global migration towards a neighborhood of the ground state and subsequent exponential convergence.

\noindent\textbf{AMS(2010) subject classifications.} 65M12, 35K20, 35K35, 35K55, 65Z05
\end{abstract}

\section{Introduction}

The nonlinear Schr\"odinger equation is a fundamentally important model in mathematical physics, widely applied in quantum physics \cite{bao2013mathematical, lieb2001rigorous}, nonlinear optics \cite{chen2006foundations, Malomed} and wave dynamics in general \cite{sulem1999nonlinear,dnls}. To describe matters in a rotating frame, e.g., the rotating BEC \cite{RBEC},  a rotating force term has been further added leading to the following rotating nonlinear Schr\"odinger equation (RNLS) \cite{aftalion2001vortices, bao2013mathematical} in $d$ dimensions ($\mathbf x=(x_1,\ldots,x_d) \in \mathbb{R}^{d},\ d\in\mathbb{N}^+$):
\begin{align*}
	\mathrm i\partial_t\psi(\mathbf x,t)
	= -\frac{1}{2}\Delta \psi(\mathbf x,t) + V(\mathbf x)\,\psi(\mathbf x,t)
	+ \beta|\psi(\mathbf x,t)|^{p-1}\psi(\mathbf x,t) - \Omega\,L_z\psi(\mathbf x,t),\quad t>0,
\end{align*}
where $\psi$ is a complex-valued wave function, $\Omega\ge 0$ denotes the rotating speed with $L_{z}$ the angular momentum operator
\begin{align*}
	L_{z} = \begin{cases}
		\mathrm i \left( x_{2} \partial_{x_{1}} - x_{1} \partial_{x_{2}} \right), & d\geq 2, \\ 0, & d=1,
	\end{cases}
\end{align*}
and $V: \mathbb{R}^{d} \to \mathbb{R}$ is an external trapping potential typically taken as a harmonic oscillator $V(\mathbf{x}) = \frac{1}{2}\sum_{i = 0}^{d}\gamma_{i}^{2} x_{i}^{2}$ with $\gamma_{i}$ the trapping frequency in applications. Moreover,  $\beta\in\mathbb{R}$ is a given parameter denoting the strength of nonlinear self-interaction with $\beta>0$ and $\beta<0$ corresponding to the defocusing and focusing cases, and $p>1$ is a given exponent.
For $d\ge 3$, one usually assumes $1<p<(d+2)/(d-2)$ for well-posedness of the model \cite{berestycki1983nonlinearI, pohozaev1965eigenfunctions}. 

Among all possible solutions of RNLS, the standing wave or soliton solution  $\psi(\mathbf{x},t) = \mathrm{e}^{\mathrm{i}\omega t}\phi(\mathbf{x})$ ($\omega \in \mathbb{R}$ often referred as the chemical potential) that maintains its shape during dynamics, finds special importance in applications, where $\phi(\mathbf{x})$ solves the stationary RNLS:
\begin{align}
	-\frac{1}{2}\Delta \phi(\mathbf{x})
	+ V(\mathbf{x})\,\phi(\mathbf{x})
	+\beta\bigl|\phi(\mathbf{x})\bigr|^{p-1}\phi(\mathbf{x}) - \Omega L_{z}\phi(\mathbf{x})
	+ \omega\,\phi(\mathbf{x})
	= 0, \quad 
	\mathbf x \in \mathbb{R}^{d}.
	\label{eq:semilinear-elliptic-rot}
\end{align}
It has been known from elliptic theory that \eqref{eq:semilinear-elliptic-rot} possesses infinitely many nontrivial solutions \cite{berestycki1983nonlinearII, strauss1977existence}. Identifying the most stable and physically relevant standing waves corresponds to the study of ground state (GS), where two types of GS have been defined in the literature \cite{berestycki1983nonlinearI, bao2013mathematical,Dovetta}. The first is the \emph{energy GS} that is a minimizer of the energy functional \cite{bao2013mathematical}
\begin{equation*}
    E_{\Omega}(\phi) := \frac{1}{2}\|\nabla \phi\|_{L^{2}}^{2} + \int_{\mathbb{R}^d} V|\phi|^2\,{\rm d}\mathbf{x}
    + \frac{2\beta}{p+1}\|\phi\|_{L^{p+1}}^{p+1} + L_{\Omega}(\phi),
    \qquad
    L_{\Omega}(\phi) := -\Omega\int_{\mathbb{R}^d}{ \overline{\phi} L_{z}\phi }\,\text{d}\mathbf{x},
\end{equation*}
subject to a prescribed mass $m_{0}>0$, i.e., $\min\{E_{\Omega}(\phi)\,:\, \phi\in H^{1}(\mathbb{R}^d),\ \|\phi\|_{L^{2}} = m_{0}\}$.
The second kind is the \emph{action GS} defined as a nontrivial minimizer of the \textit{action functional}
\begin{equation}\label{eq:action-functional-S}
    S_{\Omega,\omega}(\phi) := \frac{1}{2}\|\nabla \phi\|_{L^{2}}^{2} + \int_{\mathbb{R}^d} V|\phi|^2\,{\rm d}\mathbf{x}
    + \frac{2\beta}{p+1}\|\phi\|_{L^{p+1}}^{p+1} + L_{\Omega}(\phi)+ \omega \|\phi\|_{L^{2}}^{2}
    =E_{\Omega}(\phi)+ \omega\|\phi\|_{L^2}^2
\end{equation}
under the natural constraint of \eqref{eq:semilinear-elliptic-rot}, receiving considerable attention ever since the work of Berestycki and Lions \cite{berestycki1983nonlinearI}, and it is later found equivalently to be defined on the Nehari manifold \cite{fukuizumi2001stability,ardila2021global,liu2023computing}, 
i.e., 
\begin{equation}\label{eq:action-gs-def-org}
    \phi_{g} \in \arg\min\{S_{\Omega, \omega}(\phi)\,:\, \phi\in \mathcal{M}\}, 
    \qquad 
    \mathcal{M} := \left\{ \phi\in X\setminus\{0\} \,\colon\, K_{\Omega, \omega}(\phi)=0 \right\},
\end{equation}
where
$$
    K_{\Omega, \omega}(\phi) = \frac{1}{2}\|\nabla \phi\|_{L^{2}}^{2} + \int_{\mathbb{R}^d} V|\phi|^2\,{\rm d}\mathbf{x} + \beta\|\phi\|_{L^{p+1}}^{p+1} + L_{\Omega}(\phi) + \omega \|\phi\|_{L^{2}}^{2}.
$$
The energy GS has been extensively studied both theoretically and numerically, e.g., \cite{RGS,lieb2001rigorous, Jeanjean, bao2004computing, bao2013mathematical, zhuang2019efficient,altmann2021jmethod, cances2010numerical, dion2007ground,antoine2017efficient, danaila2010new, danaila2017computation, henning2020sobolev, feng2026mass,liu2021normalized,CDLX2023JCP}. For action GS, the seminal work \cite{berestycki1983nonlinearI}  initiated a vast literature on theoretical studies, e.g., \cite{ardila2021global, berestycki1983nonlinearI, fukuizumi2001stability, fukuizumi2003instability, shatah1985instability, Dovetta}. Some recent efforts have been made to understand the mathematical relation between the two kinds of GS. The results \cite{Dovetta, Jeanjean,liu2025action} have revealed the in-equivalence between the energy GS and the action GS, underscoring the necessity of studying them independently.

On the computational side of the action GS, for the focusing case ($\beta < 0$), \cite{WangChunshan} employed a projected gradient flow approach to solve the Nehari-constrained minimization problem \eqref{eq:action-gs-def-org}.  Our previous work \cite{liu2023computing} proposed an $L^{p+1}$-constrained minimization formulation equivalent to but simpler than \eqref{eq:action-gs-def-org}, based on which a normalized gradient flow method was proposed. In a parallel recent work \cite{liu2026convergence}, we established the convergence of the method. On the other hand, for the defocusing case ($\beta>0$), we have shown in \cite{liu2023computing} that the constraint in \eqref{eq:action-gs-def-org} can be completely removed, leading to an equivalent unconstrained minimization formulation for the action GS:
\begin{equation}\label{eq:action-gs-def}
    \phi_{g} \in \arg\min\{S_{\Omega, \omega}(\phi)\,:\, \phi\in H^{1}(\mathbb{R}^{d})\}.
\end{equation}
Based on \eqref{eq:action-gs-def}, 
an efficient direct gradient flow (DGF) method has been proposed. Despite its practical effectiveness demonstrated in \cite{liu2023computing, liu2025computing}, the DGF scheme still lacks theoretical support in two major aspects: (i) only a \emph{modified} action is shown \cite{liu2023computing} to decay without a guaranty for stability of the original action $S_{\Omega,\omega}$; (ii) no convergence analysis has been done so far.

This work focuses on the defocusing action GS of RNLS and is targeted to fill the aforementioned theoretical gaps for the DGF scheme. More precisely, we shall first establish for the DGF scheme an \textbf{unconditional stability in the {original} action}. That is, $S_{\Omega,\omega}$ is monotonically non-increasing along the discrete flow for arbitrary time step. Then, we shall carry out a rigorous convergence analysis for the DGF scheme and obtain two main results: (i) under minor assumptions, \textbf{a global convergence}; (ii) under a reasonable non-degeneracy assumption, a {local} convergence with the \textbf{optimal exponential rate} toward the action GS. The global and local behavior of DGF will be numerically justified: the numerical solution first migrates into a neighborhood of the GS, and subsequently exhibits exponential convergence. The novelty of our analysis includes: 
\begin{itemize}
    \item The uniform boundedness of sublevel set $S_{\leq m} = \left\{ \phi\in H^{1}(\mathbb{R}^{d}) \colon S_{\Omega, \omega}(\phi)\leq m \right\}$ in $H^1$ for any $m>0$, which gives stability of the discrete trajectory in DGF and is crucial for establishing the monotonic decay of $S_{\Omega,\omega}$.

    \item A tailored $H^{1}$-distance between phase-shift equivalence classes, which is introduced to handle the invariance of $S_{\Omega,\omega}$ for complex-valued GS under constant phase changes, enabling measurement of the distance from the GS in the presence of quantized vortices. 

    \item A new analytical framework for establishing exponential convergence rate of time-discrete gradient flow  toward the steady state.
\end{itemize}

The rest of the paper is organized as follows. Section~\ref{sec. 2} briefly reviews the DGF method and establishes the stability of its action functional. In Section~\ref{sec. 3}, we in a sequel establish the global convergence result and the local exponential convergence rate of  DGF. Section~\ref{sec. 4} presents  numerical experiments that confirm our theoretical findings. Finally, Section~\ref{sec. 5} provides some concluding remarks. For later use, we clarify some notation and conventions: $a\lesssim b$ means $|a|\le C\, b$ and $a\gtrsim b$ means $|a|\ge C\, b$, for some constant $C>0$ independent of temporal discretization parameters (such as the time step $\tau$ and the time level $n$); $a\asymp b$ means both $a\lesssim b$ and $b\lesssim a$. Throughout the work, we assume that $\beta>0$ in \eqref{eq:semilinear-elliptic-rot} or \eqref{eq:action-functional-S}.

\section{DGF method and its stability} \label{sec. 2}

This section begins with a brief review of the DGF scheme proposed in \cite{liu2023computing} for computing the action GS in the defocusing case, and then establishes in detail its stability in terms of the unconditional decay of the action functional $S_{\Omega,\omega}$. 

\subsection{Preliminaries} 
Let us first give some notations. 
For numerical computation and theoretical analysis, we will consider the problem  \eqref{eq:action-gs-def-org} or \eqref{eq:action-gs-def} on a truncated bounded domain $D \subset \mathbb{R}^d$ with appropriate boundary conditions (e.g., homogeneous Dirichlet or periodic boundary conditions). Such a truncation is justified as the GS exhibits a fast decay in the far field at the presence of the confinement potential $V(\mathbf{x})$, where by choosing $D$ sufficiently large, the truncation error is negligible. The real inner products of the functional spaces $L^{2}(D)$ and $H^{1}(D)$ are denoted as 
\[
	\left\langle \phi, \psi \right\rangle_{L^{2}(D)} := \operatorname*{Re}\left( \int_{D}{ \phi \, \overline{\psi} }\,\text{d}\mathbf{x} \right), \quad \left\langle \phi, \psi \right\rangle_{H^{1}(D)} := \operatorname*{Re}\left( \int_{D}{ \phi \, \overline{\psi} }\,\mathrm{d}\mathbf{x} + \int_{D}{ \nabla\phi \cdot \nabla\overline{\psi} }\,\mathrm{d}\mathbf{x} \right),
\]
where $\overline{z}$ means the complex conjugate of $z\in \mathbb{C}$, and their induced norms are denoted as $\left\| \cdot \right\|_{L^{2}}$
and $\left\| \cdot \right\|_{H^{1}}$. We will write $X = H^{1}(D)$ for short and also denote
\[ 
\left\langle \phi, \psi \right\rangle := \int_{D}{ \phi \, \overline{\psi} }\,\mathrm{d}\mathbf{x}\quad  \phi\in H^{-1}(D),\ \psi\in X,
\]
where $H^{-1}(D)$ is the dual space of $X$ with the norm 
$$\left\| \phi \right\|_{H^{-1}} = \sup_{\psi\in X, \psi \neq 0} \frac{|\left\langle \phi, \psi \right\rangle|}{\left\| \psi \right\|_{H^{1}}}.$$

Hereafter, we change the domain of integration in functionals $S_{\Omega,\omega}$, $K_{\Omega,\omega}$, $L_{\Omega}$, $E_{\Omega}$ and in some relevant norms from $\mathbb{R}^d$ to $D$.  
To provide the existence of GS, we are under the following basic setup.

\begin{assumption}\label{assump:potential-condition}
	Let $1 < p < \frac{d+2}{d-2}$ for $d \geq 3$ and $1 < p < \infty$ for $d = 1, 2$. Assume that $V \in L^{\infty}(D)$, and one of the following holds:
	\begin{enumerate}[label=(\alph*)]
		\item $d=1$ (or $\Omega = 0$) and $V(\mathbf{x}) \geq 0$ for all $\mathbf{x} \in D$.
		\item $d\ge 2$, $V(\mathbf{x}) = \frac{1}{2}\sum_{i = 0}^{d}\gamma_{i}^{2} x_{i}^{2}$ with $\gamma_{i} > 0$ and $\Omega < \min\left\{ \gamma_{1}, \gamma_{2} \right\}$.
	\end{enumerate}
\end{assumption}

Under Assumption~\ref{assump:potential-condition}, $X$ is (compactly) embedded into $ L^{p+1}(D)$, and the functionals $ S_{\Omega, \omega} $ and $K_{\Omega, \omega}$ are well-defined in $X$. Next, we recall several standard inequalities and fix the notation for the involved generic constants that will be used throughout.

\begin{lemma}\label{lem:standard-inequalities}
	For any $u \in X$, the following inequalities hold:
	\begin{enumerate}[label=(\roman*)]
		\item (Embedding inequality) There exists a constant $C_{E} := C_{E}(d, p, D)$ such that
		\begin{align*}
			\| u \|_{L^{p+1}} \leq C_{E} \| u \|_{H^{1}}.
		\end{align*}
		\item (Gagliardo--Nirenberg inequality) There exists a constant $C_{G} := C_{G}(d, p, D)$ such that
		\begin{align*}
			\| u \|_{L^{p+1}} \leq C_{G} \| u \|_{L^{2}}^{1-\theta} \| \nabla u \|_{L^{2}}^{\theta},
		\end{align*}
		where $\theta = d\left( \frac{1}{2} - \frac{1}{p+1} \right) \in (0,1)$.
		\item If $d\ge 2$, then for any constant $c>0$, 
		\begin{align*}
			\left| \Omega \int_{D}{ \overline{u}\, L_{z}u }\,\text{d}\mathbf{x} \right| \leq \int_{D}{ \left( \frac{c}{2}\left| \nabla u \right|^{2} + \frac{|\Omega|^{2}}{2c}\left( x_{1}^{2} + x_{2}^{2} \right)|u|^{2} \right) }\,\text{d}\mathbf{x}.
		\end{align*}
	\end{enumerate}
\end{lemma}

When $V$ satisfies Assumption~\ref{assump:potential-condition}, the linear operator $-\frac{1}{2}\Delta + V - \Omega L_{z} $ has a countably infinite discrete set of eigenvalues with the smallest eigenvalue $\lambda_0$ being simple \cite{GT}: 
\begin{align}
	\lambda_0 := \inf_{u\in X, \, \|u\|_{L^2}=1 }\left( \frac{1}{2}\|\nabla u\|_{L^2}^2 + \int_{D}V|u|^2\,{\rm d}\mathbf{x} + L_{\Omega}(u) \right). \label{eq:omega-lambda-0-def}
\end{align}
When the given chemical potential in \eqref{eq:action-functional-S} satisfies $\omega < -\lambda_0$, we have  established in \cite{liu2023computing} the existence of a global minimizer and shown that the Nehari constraint can be removed. Define the sublevel set $S_{\leq m} := \left\{ \phi\in X \colon S_{\Omega, \omega}(\phi)\leq m \right\}$ for any $m \ge 0$.

\begin{proposition}[Unconstrained formulation \cite{liu2023computing}]\label{lem:GS-existence}
	Let Assumption~\ref{assump:potential-condition} hold and $\omega < -\lambda_0$. Then, the following properties hold:
	\begin{enumerate}[label=(\roman*)]
		\item The action functional $S_{\Omega, \omega}$ is bounded from below, i.e., $\inf_{\phi\in X} S_{\Omega, \omega}(\phi) > -\infty$.
		
		\item  $\mathcal{M} \subset S_{\leq 0}$ and $\inf_{\phi \in \mathcal{M}} S_{\Omega,\omega}(\phi) < 0$.
		
		\item There exists a $\phi_{g}\in X$ such that
		\begin{align}\label{gs unconstrain def}
			S_{\Omega, \omega}(\phi_{g}) = \inf_{\phi\in X} S_{\Omega, \omega}(\phi) = \inf_{\phi\in \mathcal{M}} S_{\Omega, \omega}(\phi).
		\end{align}
	\end{enumerate}
\end{proposition}
The statements above are stated in \cite{liu2023computing} for the whole space $\mathbb{R}^{d}$, which hold here also for the case of bounded domain $D$ under subtle modifications. Moreover, we can have a boundedness result in $X$.

\begin{lemma}[Boundedness of sublevel set $S_{\leq m}$ in $H^1$]\label{lem:uniform-bound-M}
Let Assumption~\ref{assump:potential-condition} hold and $\omega < -\lambda_0$.
Then for any $m\geq 0$, there exists some $M>0$ such that for all $\phi \in S_{\leq m}$, we have $\left\| \phi \right\|_{H^{1}} \leq M$.
\end{lemma}
\begin{proof}
	First, we consider case (a) of  Assumption~\ref{assump:potential-condition}. For any $\phi \in S_{\leq m}$, by using H\"older's inequality, we have for some constant $C>0$,
	\begin{align*}
		m \ge S_{\Omega, \omega} (\phi) 
        \geq (\lambda_0+\omega)\|\phi\|_{L^2}^2+\frac{2\beta}{p+1}\|\phi\|_{L^{p+1}}^{p+1}
        \geq (\lambda_0+\omega)\left\| \phi \right\|_{L^{2}}^{2} + C \beta \left\| \phi \right\|_{L^{2}}^{p+1},
	\end{align*}
	which implies straightforwardly the uniform boundedness of $\|\phi\|_{L^{2}}$ due to $p>1$ and $C\beta>0$. 
    Then, using the fact that $S_{\Omega,\omega}(\phi)\le m$ with $\Omega=0$ and $\beta>0$ gives the uniform bound for $\left\| \nabla\phi \right\|_{L^{2}}$ and $\left\| \phi \right\|_{H^{1}}$. 
	
	We now consider case (b) of Assumption~\ref{assump:potential-condition}. From the result of Lemma~\ref{lem:standard-inequalities} \textit{(iii)}, we find
	\begin{align*}
		m \ge&\, S_{\Omega, \omega} (\phi) \ge \frac{1-\delta}{2} \left\| \nabla\phi \right\|_{L^{2}}^{2} + \int_{D}{ \left( (V + \omega)\left| \phi \right|^{2} +  \frac{2\beta}{p+1}|\phi|^{p+1} - \frac{|\Omega|^{2}}{2\delta}\left( x_{1}^{2} + x_{2}^{2} \right)|\phi|^{2} \right) }\,\text{d}\mathbf{x}  \\
		\ge&\, \frac{1-\delta}{2} \left\| \nabla\phi \right\|_{L^{2}}^{2} + \int_{D}{ \left( (V_{\Omega, \delta} + \omega)\left| \phi \right|^{2} +  \frac{2\beta}{p+1}|\phi|^{p+1} \right) }\,\text{d}\mathbf{x},
	\end{align*}
	with $V_{\Omega, \delta} = \frac{1}{2}\sum_{i=1}^{d}\gamma_{i}^{2}x_{i}^{2} - \frac{|\Omega|^{2}}{2\delta}\left( x_{1}^{2} + x_{2}^{2} \right)$. Choosing $\left( |\Omega| / \min\{\gamma_{1},\gamma_{2}\} \right)^{2}< \delta < 1$, we have $V_{\Omega,\delta}(\mathbf{x}) \ge 0$ and then the assertion can be proved in the same manner for case (a).
\end{proof}

Proposition~\ref{lem:GS-existence} indicates that the defocusing action GS can be obtained by minimizing the action functional $S_{\Omega, \omega}$ directly over the entire space $X$, without resorting to the Nehari manifold constraint \eqref{eq:action-gs-def-org}. Consequently, the problem reduces to an unconstrained minimization of \eqref{eq:action-functional-S}, which can be approached via a direct gradient flow (DGF):
\begin{align}
	\partial_{t}\phi = - \frac{\delta S_{\Omega, \omega}}{\delta \overline{\phi}} =  - \left( -\frac{1}{2}\Delta + V + \beta |\phi|^{p-1} - \Omega L_{z} + \omega  \right) \phi := -H_{\phi}\phi, \quad t\geq 0,  \label{eq:GF-continuous}
\end{align}
starting with an initial guess $\phi(\cdot,0) = \phi_{0} \in X\setminus\{ 0\}$. In such a way, it is straightforward to deduce from \eqref{eq:GF-continuous} that
\begin{align*}
	\frac{{\rm d}}{{\rm d} t} S_{\Omega, \omega}(\phi(\cdot,t)) = -2\|\partial_{t} \phi(\cdot,t)\|_{L^{2}}^{2} \le 0,\quad \forall t \geq 0.
\end{align*}
Various discretization techniques can be applied to the gradient flow \eqref{eq:GF-continuous}, and \cite{liu2023computing} suggests the following backward-forward Euler scheme: with the initial guess $\phi^{0} \in X$,
\begin{subequations}\label{eq:GF-semi-discretization}
	\begin{align}
		&\frac{1}{\tau}(\phi^{n+1}-\phi^{n}) = -\mu^{n+1}, \quad n\geq0,\label{eq_a:GF-semi-discretization} \\ 
		&\mu^{n+1} = \left( -\frac{1}{2}\Delta + \alpha \right)\phi^{n+1} + \left( V + \omega + \beta |\phi^{n}|^{p-1} - \Omega L_{z} - \alpha  \right) \phi^{n}, \label{eq_b:GF-semi-discretization}
	\end{align}
\end{subequations}
where $\tau  > 0$ is the time step and $\alpha \geq 0$ is a stabilization parameter chosen to enhance stability performance. 

In \cite{liu2023computing}, by choosing $\alpha$ properly, the monotone decay of the action functional was numerically observed but not rigorously established, where only a modified action was proved to decay during time stepping. We will fill the gap in the following.

\subsection{Stability in action functional}
In order to study the stability of $S_{\Omega,\omega}$ associated with the DGF scheme \eqref{eq:GF-semi-discretization}, we first derive several a priori bounds for $\phi^{n+1}$. Rearranging \eqref{eq:GF-semi-discretization}, we obtain
\begin{align}
	\mathcal{A}\phi^{n+1}
    = \left( 1+ \tau\alpha \right) \phi^{n} - \tau g(\phi^{n}), \label{eq:I-rearrange}
\end{align}
where $\mathcal{A} :=\left( 1 + \tau\alpha \right) I - \frac{\tau}{2}\Delta$ and $g(\phi^{n}):=\left( V + \omega + \beta |\phi^{n}|^{p-1} - \Omega L_{z} \right)\phi^{n}$.

Throughout this paper, under a given $\phi^{0}$, we set $S_{\Omega,\omega}(\phi^{0}) \leq m$ for a constant $m \ge 0$.

\begin{lemma}\label{lem:phinplus1-bound-by-phin}
Let Assumption~\ref{assump:potential-condition} hold. If $\alpha\geq\frac{1}{8}$ and $\phi^{n}\in X$, then for some $C_{1}, C_{2} > 0$,
\begin{align}
	\left\| \phi^{n+1} \right\|_{H^{1}} 
	\leq C_1\left\| \phi^{n} \right\|_{H^{1}} + C_2 \left\| \phi^{n} \right\|_{H^{1}}^p. \label{eq:phinplus1-bound-by-phin}
\end{align}
\end{lemma}
\begin{proof}
By \eqref{eq:I-rearrange}, we can write $\phi^{n+1}=\phi_1-\tau\phi_2$, where $\phi_1$ and $\phi_2$ are determined, respectively, by the elliptic equations $\mathcal{A}\phi_1=\left( 1 + \tau\alpha \right) \phi^{n}$ and $\mathcal{A}\phi_2=g(\phi^{n})$ (with the same boundary conditions as in \eqref{eq:GF-semi-discretization}). 
On the one hand, noting that $\phi^n=\left(1 + \tau\alpha \right)^{-1}\mathcal{A}\phi_1=\phi_1-c\Delta\phi_1$ with $c:=\frac{\tau}{2(1+\tau\alpha)}>0$, one has
\begin{align*}
    \left\| \phi^{n} \right\|_{H^{1}}^2
    =\left\| \phi_1-c\Delta\phi_1\right\|_{H^{1}}^2
    =\left\| \phi_1\right\|_{H^{1}}^2+2c\left(\left\| \nabla\phi_1\right\|_{L^{2}}^2+\left\| \Delta\phi_1\right\|_{L^{2}}^2\right) + c^2\left\|\Delta\phi_1\right\|_{H^{1}}^2
    \geq \left\| \phi_1\right\|_{H^{1}}^2.
\end{align*}
On the other hand, taking the $L^2$ inner product of $\mathcal{A}\phi_2=g(\phi^{n})$ with $\phi_2$ gives
\begin{align*}
    \left( 1 + \tau\alpha \right)\left\| \phi_2\right\|_{L^{2}}^2+\frac{\tau}{2}\left\|\nabla \phi_2\right\|_{L^{2}}^2=
    \left\langle \mathcal{A}\phi_2,\phi_2\right\rangle =\left\langle g(\phi^n),\phi_2\right\rangle 
    \lesssim \left(\|\phi^n\|_{L^2}+\beta\,\|\phi^n\|_{L^{p+1}}^p\right)\|\phi_2\|_{H^1}.
\end{align*}
Since $\alpha\geq \frac{1}{8}$, one gets $\tau\left\| \phi_2\right\|_{H^{1}}\lesssim \|\phi^n\|_{L^2}+\beta \|\phi^n\|_{L^{p+1}}^p\lesssim \|\phi^n\|_{H^1}+\beta \|\phi^n\|_{H^{1}}^p$. 
Consequently, 
\begin{align*}
	\left\| \phi^{n+1} \right\|_{H^{1}} \leq \left\|\phi_1\right\|_{H^{1}} + \tau \left\|\phi_2\right\|_{H^{1}}
    \lesssim \left\| \phi^{n} \right\|_{H^{1}} + \left\| \phi^{n} \right\|_{H^{1}}^p.
\end{align*}
This completes the proof.
\end{proof}
According to Lemmas~\ref{lem:uniform-bound-M} and \ref{lem:phinplus1-bound-by-phin}, if $\phi^{n}\in S_{\leq m}$, then $\left\| \phi^{n} \right\|_{H^{1}} \leq M$ and
\begin{align}
	\left\| \phi^{n+1} \right\|_{H^{1}} 
    \leq C_1M+C_2M^p=:C_M. \label{eq:phi-bound-by-CM}
\end{align}
Now, we define for $d=1$ or $\Omega = 0$,
\begin{align*}
	\alpha_{M} = \frac{1}{2}\max\Big\{0,\ \mathop{\mathrm{ess\,sup}}_{\mathbf{x}\in D}\big( V(\mathbf{x})+\omega \big)\Big\} + \frac{1}{2}(1-\theta)(8\theta)^{\frac{\theta}{1-\theta}} \left( \beta\,p\, C_{E}^{p-1}C_M^{p-1}C_{G}^2 \right)^{\frac{1}{1-\theta}},
\end{align*}
otherwise, 
\begin{align*}
	\alpha_{M} = \frac{1}{2}\max\Big\{0,\ \mathop{\mathrm{ess\,sup}}_{\mathbf{x}\in D}\big( V(\mathbf{x})+\omega + 2|\Omega|^2(x_{1}^{2} + x_{2}^{2}) \big)\Big\} +\frac{1}{2}(1-\theta)(8\theta)^{\frac{\theta}{1-\theta}} \left( \beta\,p\, C_{E}^{p-1}C_M^{p-1}C_{G}^2 \right)^{\frac{1}{1-\theta}}.
\end{align*}
The action-functional stability result then reads as follows.

\begin{theorem}[Unconditional decay of action]\label{thm:action_decaying}
	Let Assumption~\ref{assump:potential-condition} hold. For a given $\phi^{0}\in S_{\leq m}$, suppose that $\phi^{n}\in X$ and choose $\alpha \geq \alpha_{M} + \frac{1}{8}$. Then, the DGF scheme \eqref{eq:GF-semi-discretization} possesses the unconditional decaying property for the action functional $S_{\Omega, \omega}$: $\forall\tau>0$,
	\begin{align}
		S_{\Omega, \omega}(\phi^{n+1}) - S_{\Omega, \omega}(\phi^{n}) \le -2\tau\left\| \mu^{n+1} \right\|_{L^{2}}^{2} - \frac{\tau^{2}}{4}\left\| \mu^{n+1} \right\|_{H^{1}}^{2}. \label{eq:S_decaying}
	\end{align}
\end{theorem}

\begin{proof}
	By Lemma~\ref{lem:uniform-bound-M}, we have $\left\| \phi^{0} \right\|_{H^{1}} \leq M$. We move on in an induction manner. Suppose that
	\begin{align}
		S_{\Omega, \omega}(\phi^{k}) \leq m, \quad 0 \leq k \leq n, \label{eq:induction-assum}
	\end{align}
	which implies $\left\| \phi^{k} \right\|_{H^{1}} \leq M$ for all $0 \leq k \leq n$, and by \eqref{eq:phi-bound-by-CM} we have $\left\| \phi^{n+1} \right\|_{H^{1}} \leq C_{M}$. 
Taking the real $L^{2}$ inner products of \eqref{eq_a:GF-semi-discretization} and \eqref{eq_b:GF-semi-discretization} with $\tau\mu^{n+1}$ and $\phi^{n+1} - \phi^{n}$, respectively, yields
\begin{align*}
    -2\left(\tau+\alpha\tau^2\right)\left\|\mu^{n+1}\right\|_{L^2}^2
    &=\frac{1}{2}\left\|\nabla\phi^{n+1}\right\|_{L^2}^2-\frac{1}{2}\left\|\nabla\phi^{n}\right\|_{L^2}^2+\frac{1}{2}\left\|\nabla\left(\phi^{n+1}-\phi^n\right)\right\|_{L^2}^2 \\
    &\quad\; +\left\langle V+\omega,\, \left|\phi^{n+1}\right|^2-\left|\phi^{n}\right|^2-\left|\phi^{n+1}-\phi^n\right|^2 \right\rangle_{L^2(D)} \\
    &\quad\; +\, L_{\Omega}\left(\phi^{n+1}\right)-L_{\Omega}\left(\phi^{n}\right)-L_{\Omega}\left(\phi^{n+1}-\phi^{n}\right) \\
    &\quad\; +\, 2\beta \left\langle \left|\phi^{n}\right|^{p-1}\phi^n, \phi^{n+1}-\phi^n\right\rangle_{L^2(D)}.
\end{align*}
It follows that
\begin{align}
    S_{\Omega, \omega}(\phi^{n+1}) - S_{\Omega, \omega}(\phi^{n}) 
    =-2\left(\tau+\alpha\tau^2\right)\left\|\mu^{n+1}\right\|_{L^2}^2
    -\frac{\tau^2}{2}\left\|\nabla\mu^{n+1}\right\|_{L^2}^2
    +Q^{n+1}+R^{n+1},
    \label{eq:Somega-gap-2}
\end{align}
where
\begin{align*}
Q^{n+1} &:=\left\langle V+\omega,\,\left|\phi^{n+1}-\phi^n\right|^2 \right\rangle_{L^2(D)}+L_{\Omega}\left( \phi^{n+1} - \phi^{n} \right), \\
    R^{n+1} &:= 2\beta\int_D\left[ \frac{1}{p+1}\left|\phi^{n+1}\right|^{p+1} - \frac{1}{p+1}\left|\phi^{n}\right|^{p+1} - \mathrm{Re}\left( \left|\phi^{n}\right|^{p-1}\phi^{n}\overline{\left(\phi^{n+1}-\phi^n\right)}\right) \right]\mathrm{d}\mathbf{x}.
\end{align*}
Next, we estimate $Q^{n+1}$ and $R^{n+1}$ separately. For $Q^{n+1}$, from the statement \textit{(iii)} of Lemma~\ref{lem:standard-inequalities}, we have
\begin{align*}
	Q^{n+1} 
    &\leq \frac{\tau^2}{8} \left\| \nabla\mu^{n+1} \right\|_{L^{2}}^{2} + \tau^2\int_{D}\left(V(\mathbf{x})+\omega+2|\Omega|^{2}\left(x_{1}^{2} + x_{2}^{2}\right)\right)\left|\mu^{n+1}\right|^2 \mathrm{d}\mathbf{x} \\
    &\leq \frac{\tau^2}{8} \left\| \nabla\mu^{n+1} \right\|_{L^{2}}^{2} + \tau^2\max\Big\{0,\ \mathop{\mathrm{ess\,sup}}_{\mathbf{x}\in D}\Big(V(\mathbf{x})+\omega+2|\Omega|^{2}\left(x_{1}^{2} + x_{2}^{2}\right)\Big)\Big\} \left\| \mu^{n+1}\right\|_{L^{2}}^{2}.
\end{align*}
For $R^{n+1}$, denoting $f(\rho)=\frac{2\beta}{p+1}\int_D\left|\phi_\rho\right|^{p+1}\mathrm{d}\mathbf{x}$ with $\phi_\rho:=(1-\rho)\phi^n+\rho\phi^{n+1}$, $\rho\in[0,1]$, we have
\begin{align*}
    R^{n+1} 
    &=f(1)-f(0)-f'(0)=\int_0^1(1-\rho)f''(\rho)\mathrm{d}\rho \\
    &=\beta\int_0^1(1-\rho)\int_D\left[(p+1)\left|\phi_\rho\right|^{p-1}\left|\phi^{n+1}-\phi^n\right|^2+(p-1)\,\mathrm{Re}\Big(\left|\phi_\rho\right|^{p-3}\phi_\rho^2\overline{\left(\phi^{n+1}-\phi^n\right)}^2\Big)\right]\mathrm{d}\mathbf{x}\mathrm{d}\rho \\
    &\leq \tau^2\beta\,p\,\max_{\rho\in[0,1]}\left\|\phi_\rho\right\|_{L^{p+1}}^{p-1}\left\|\mu^{n+1}\right\|_{L^{p+1}}^2,
\end{align*}
By Lemma~\ref{lem:phinplus1-bound-by-phin} and noting that $C_M\geq M$, one has the following uniform bound for all $\rho\in[0,1]$:
\begin{align*}
\left\|\phi_\rho\right\|_{L^{p+1}}^{p-1}
\leq \max\left\{\left\|\phi^n\right\|_{L^{p+1}}^{p-1},\left\|\phi^{n+1}\right\|_{L^{p+1}}^{p-1}\right\}
\leq C_E^{p-1}\max\left\{\left\|\phi^n\right\|_{H^{1}}^{p-1},\left\|\phi^{n+1}\right\|_{H^{1}}^{p-1}\right\} 
\leq C_E^{p-1}C_M^{p-1}.
\end{align*}
Applying the Gagliardo-Nirenberg inequality and Young's inequality of the form $ ab\leq \frac{\varepsilon a^{p_1}}{p_1}+\varepsilon^{-\frac{p_2}{p_1}}\frac{b^{p_2}}{p_2}$, $\forall a,b\geq 0$, with $\frac{1}{p_1}+\frac{1}{p_2}=1$ and $\varepsilon>0$, we obtain
\begin{align*}
    R^{n+1} 
    &\leq \tau^2\beta\,p\,C_E^{p-1}C_M^{p-1}C_G^2\left\|\nabla\mu^{n+1}\right\|_{L^2}^{2\theta}\left\|\mu^{n+1}\right\|_{L^2}^{2-2\theta} \\
    &\leq \tau^2\beta\,p\,C_E^{p-1}C_M^{p-1}C_G^2\left(\theta\varepsilon\left\|\nabla\mu^{n+1}\right\|_{L^2}^{2}+(1-\theta)\varepsilon^{-\frac{\theta}{1-\theta}}\left\|\mu^{n+1}\right\|_{L^2}^{2}\right).
\end{align*}
Taking $\varepsilon = \left( 8\,\theta\, \beta\, p\, C_{E}^{p-1}C_{M}^{p-1} C_{G}^2\right)^{-1}$, we have
	\begin{align*}
		R^{n+1} \leq \frac{\tau^2}{8} \left\| \nabla \mu^{n+1} \right\|_{L^{2}}^{2} + \tau^2(1-\theta) \left( 8\theta \right)^{\frac{\theta}{1-\theta}}\left( \beta\, p\, C_{E}^{p-1}C_M^{p-1} C_{G}^2 \right)^{\frac{1}{1-\theta}} \left\| \mu^{n+1} \right\|_{L^{2}}^{2}.
	\end{align*}
Substituting the estimates of $Q^{n+1}$ and $R^{n+1}$ into \eqref{eq:Somega-gap-2} and using the definition of $\alpha_M$, we find
	\begin{align*}
		S_{\Omega, \omega}(\phi^{n+1}) - S_{\Omega, \omega}(\phi^{n}) 
        &= -2\left(\tau+\alpha\tau^2\right)\left\|\mu^{n+1}\right\|_{L^2}^2
    -\frac{\tau^2}{2}\left\|\nabla\mu^{n+1}\right\|_{L^2}^2+Q^{n+1}+R^{n+1} \\
        &\leq -2\tau\left\|\mu^{n+1}\right\|_{L^2}^2
    -\frac{\tau^2}{4}\left\|\nabla\mu^{n+1}\right\|_{L^2}^2 -2\tau^2\left(\alpha-\alpha_M\right)\left\|\mu^{n+1}\right\|_{L^2}^2.
	\end{align*}
    With the condition $\alpha\geq\alpha_M+\frac{1}{8}$  given  in Theorem~\ref{thm:action_decaying}, we can get 
	\begin{align*}
		S_{\Omega, \omega}(\phi^{n+1}) - S_{\Omega, \omega}(\phi^{n}) 
        \leq -2\tau\left\| \mu^{n+1} \right\|_{L^{2}}^{2} - \frac{\tau^{2}}{4}\left\| \mu^{n+1} \right\|_{H^{1}}^{2},
	\end{align*}
    which shows \eqref{eq:S_decaying}.
	The hypothesis \eqref{eq:induction-assum} further leads to $S_{\Omega, \omega}(\phi^{n+1}) \leq m$, closing the induction.
\end{proof}

\begin{remark}
	In practical computations, it might be difficult to precisely quantify the theoretical $\alpha_M$ in Theorem \ref{thm:action_decaying}, and the constant $1/8$ in Theorem~\ref{thm:action_decaying} is a technical factor required for the subsequent convergence-rate analysis. From the perspective of action descent alone, the condition $\alpha\geq \alpha_M+1/8$ is not strictly necessary. Motivated by \eqref{eq:Somega-gap-2} and the facts
	\begin{align*}
		\left| Q^{n+1} \right| &\leq \frac{\tau^2}{2} \left\| \nabla\mu^{n+1} \right\|_{L^{2}}^{2} + \tau^2\int_{D}\left(V(\mathbf{x})+\omega+\frac{|\Omega|^{2}}{2}\left(x_{1}^{2} + x_{2}^{2}\right)\right)\left|\mu^{n+1}\right|^2 \mathrm{d}\mathbf{x},  \\
		\left| R^{n+1} \right| & \leq \tau^2\beta\,p\int_{D}{ \max\left\{ \left|\phi^{n}\right|^{p-1}, \left|\phi^{n+1}\right|^{p-1} \right\} \left|\mu^{n+1}\right|^{2} }\,\mathrm{d}\mathbf{x} 
		\approx \tau^2\beta\,p\int_{D}{ \left|\phi^{n}\right|^{p-1} \left|\mu^{n+1}\right|^{2} }\,\mathrm{d}\mathbf{x}, 
	\end{align*}
    we consider adaptively selecting $\alpha$ at each iteration step as
	\begin{align}
		\alpha^{n} := \frac{1}{2}\max\left\{0,\ \mathop{\mathrm{ess\,sup}}_{\mathbf{x}\in D}\left( V(\mathbf{x})+\omega + \frac{|\Omega|^{2}}{2}(x_{1}^{2} + x_{2}^{2}) + \beta(p+2) |\phi^{n}|^{p-1} \right)\right\}. \label{eq:alpha_n}
	\end{align}
    In our numerical experiments, the value of $\alpha^{n}$ is typically much larger than $1/8$.
\end{remark}

\section{Convergence Analysis} \label{sec. 3}

This section analyzes the convergence of the DGF scheme \eqref{eq:GF-semi-discretization}, giving the main results of the work. We will first present a global convergence result and then establish a refined quantitative  convergence rate of DGF toward the GS.

\subsection{A global convergence result}

Our first result on convergence reads as follows.

\begin{theorem}[Accumulation point \& unconditional convergent subsequence]\label{thm:convergent-subsequences}
	For a given $\phi^{0}\in H^{2}(D)$, suppose that Assumption~\ref{assump:potential-condition} holds, $\omega < -\lambda_{0}$ and $\alpha \ge \alpha_{M} + \frac{1}{8}$. Additionally, assume for $d\geq 3$ that $1<p\leq d/(d-2)$. Then, for every fixed $\tau>0$:
	\begin{enumerate}[label=(\roman*)]
		\item 
		The solution $\left\{ \phi^{n} \right\}_{n=0}^{\infty}$ of \eqref{eq:GF-semi-discretization} has a strongly convergent subsequence $\left\{ \phi^{n_{j}} \right\}_{j=0}^{\infty}$ in $H^{1}(D)$;
		
		\item Every accumulation point $\phi^{\star}\in X$ of $\left\{ \phi^{n} \right\}_{n=0}^{\infty}$ in $H^{1}(D)$ satisfies the Euler-Lagrange equation of \eqref{eq:action-gs-def-org} in the weak sense, i.e.,
		\begin{align}\label{eq:Euler-LagrangeEq-weakform}
			\left\langle H_{\phi^{\star}}\phi^{\star}, v \right\rangle = 0, \quad \forall v\in H^{1}(D),
		\end{align}
		and all accumulation points possess the same  action value $S_{\Omega, \omega}(\phi^\star)=\inf_{n\geq0}S_{\Omega,\omega}(\phi^n)$.
	\end{enumerate}
\end{theorem}
\begin{proof}
	Taking the $L^{2}$ inner product of \eqref{eq:I-rearrange} with $-\Delta \phi^{n+1}$ yields
	\begin{align}
		(1 + \tau \alpha) \left\| \nabla \phi^{n+1} \right\|_{L^{2}}^{2} + \frac{\tau}{2} \left\| \Delta \phi^{n+1} \right\|_{L^{2}}^{2} =&\, \left\langle (1 + \tau \alpha)\phi^{n} - \tau g(\phi^{n}), -\Delta\phi^{n+1} \right\rangle_{L^{2}(D)} \nonumber\\
		\leq&\, \Big( (1 + \tau \alpha) \left\| \phi^{n} \right\|_{L^{2}} + \tau \left\| g(\phi^{n}) \right\|_{L^{2}} \Big) \left\| \Delta\phi^{n+1} \right\|_{L^{2}}. \label{eq:1-H2-step1}
	\end{align}
	From the definition of $g(\phi^{n})$ and the Sobolev embedding $H^{1}(D) \hookrightarrow L^{2p}(D)$ (which holds under the assumption on $p$), we obtain
	\begin{align}
		\left\| g(\phi^{n}) \right\|_{L^{2}} \lesssim \left\| \phi^{n} \right\|_{H^{1}} + \left\| \phi^{n} \right\|_{L^{2p}}^{p} \lesssim \left\| \phi^{n} \right\|_{H^{1}} + \left\| \phi^{n} \right\|_{H^{1}}^{p}. \label{eq:g-H2-step1}
	\end{align}
	Inserting \eqref{eq:g-H2-step1} into \eqref{eq:1-H2-step1}, we get
	\begin{align*}
		\frac{\tau}{2} \left\| \Delta \phi^{n+1} \right\|_{L^{2}}^{2} \lesssim \Big( (1+\tau) \left\| \phi^{n} \right\|_{H^{1}} + \tau \left\| \phi^{n} \right\|_{H^{1}}^{p} \Big) \left\| \Delta \phi^{n+1} \right\|_{L^{2}},
	\end{align*}
	which leads to $\left\| \Delta \phi^{n+1} \right\|_{L^{2}} \lesssim (1 + \tau^{-1}) \left\| \phi^{n} \right\|_{H^{1}} + \left\| \phi^{n} \right\|_{H^{1}}^{p}$. This also means that for all $n\in \mathbb{N}^{+}$, $\phi^{n} \in H^{2}(D)$ and $\left\| \phi^{n} \right\|_{H^{2}} \leq C(\tau, M)$, where $C(\tau, M)$ depends on $\tau$ and $M$. By the Banach-Alaoglu theorem, there exists a weakly convergent subsequence $\{\phi^{n_{j}}\}_{j=0}^{\infty}$ in $H^2(D)$, and the compact embedding of $H^{2}(D)\hookrightarrow  H^{1}(D)$ implies the strong convergence of $\{\phi^{n_{j}}\}_{j=0}^{\infty}$ in the $H^{1}$-norm. This establishes the first assertion of Theorem~\ref{thm:convergent-subsequences}.
	
	Now we turn to the second assertion of Theorem~\ref{thm:convergent-subsequences}. Let $\phi^\star\in X$ be an accumulation point of $\{\phi^{n}\}_{n=0}^{\infty}$ with $\{\phi^{n_{j}}\}_{j=0}^{\infty}$ a subsequence converging strongly to $\phi^\star$ in the $H^{1}$-norm. Denoting 
	\begin{align*}
		a\left( \phi, \psi \right) := \left\langle H_{\phi} \phi, \psi \right\rangle =  \int_{D}{ \left[ \frac{1}{2}\nabla \phi \cdot \nabla \overline{\psi} + \left( V + \omega \right)\phi \, \overline{\psi} + \beta |\phi|^{p-1}\phi \, \overline{\psi} - \Omega \overline{\psi} L_{z} \phi  \right] }\,\text{d}\mathbf{x},
	\end{align*}
	we can find that for all $\psi\in X$,
	\begin{align}
		a\left( \phi^{n_{j}}, \psi \right) - a(\phi^{\star}, \psi) =&\, \int_{D}{ \left[ \frac{1}{2} \nabla \left( \phi^{n_{j}} - \phi^{\star} \right) \cdot \nabla \overline{\psi} + \overline{\psi}\,(V + \omega - \Omega L_{z}) \left( \phi^{n_{j}} - \phi^{\star} \right)   \right] }\,\text{d}\mathbf{x} \nonumber\\
		&\, + \beta \int_{D}{ \left( |\phi^{n_{j}}|^{p-1}\phi^{n_{j}} - |\phi^{\star}|^{p-1}\phi^{\star} \right) \overline{\psi} }\,\text{d}\mathbf{x}. \label{eq:a-phinj-phistar}
	\end{align}
	Note that
	\begin{align*}
		|\phi^{n_{j}}|^{p-1}\phi^{n_{j}}-|\phi^{\star}|^{p-1}\phi^{\star}
		=\int_{0}^{1}\left(\frac{p+1}{2}|\phi_{\rho}|^{p-1}(\phi^{n_{j}}-\phi^{\star}) + \frac{p-1}{2}|\phi_{\rho}|^{p-3}\phi_{\rho}^2\,\overline{(\phi^{n_{j}}-\phi^{\star})}\right)\mathrm{d}\rho,
	\end{align*}
	where $\phi_{\rho}:=\rho \phi^{n_{j}}+(1-\rho)\phi^{\star}$. Thus, we can obtain
	\begin{align*}
		\left| \int_{D}{ \left( |\phi^{n_{j}}|^{p-1}\phi^{n_{j}} - |\phi^{\star}|^{p-1}\phi^{\star} \right) \overline{\psi} }\,\text{d}\mathbf{x} \right| &\lesssim  p\, \max_{0\leq\rho\leq 1} \left\| \phi_{\rho} \right\|_{L^{p+1}}^{p-1}\left\| \phi^{n_{j}}-\phi^{\star} \right\|_{L^{p+1}} \left\| \psi \right\|_{L^{p+1}} \\
		&\lesssim \left\| \phi^{n_{j}}-\phi^{\star} \right\|_{H^{1}} \left\| \psi \right\|_{H^{1}}.
	\end{align*}
	Substituting this estimate into \eqref{eq:a-phinj-phistar} yields
	\begin{align*}
		\big| a\left( \phi^{n_{j}}, \psi \right) - a\left( \phi^{\star}, \psi \right) \big|
		\lesssim &\  \|\phi^{n_{j}} - \phi^{\star}\|_{H^{1}} \|\psi\|_{H^{1}} + \left|\int_{D}{ \left(|\phi^{n_{j}}|^{p-1}\phi^{n_{j}}-|\phi^{\star}|^{p-1}\phi^{\star} \right)\,\overline{\psi} }\,\mathrm{d}\mathbf{x}\right| \\
		\lesssim &\ \|\phi^{n_{j}} - \phi^{\star}\|_{H^{1}} \|\psi\|_{H^{1}}.
	\end{align*}
	The $H^1$-strong convergence of the subsequence $\{\phi^{n_{j}}\}_{j=0}^{\infty}$ implies that 
	for all $\psi \in X$, 
	\begin{align*}
		\lim_{j\rightarrow\infty} a\left( \phi^{n_{j}}, \psi \right) = a\left( \phi^{\star}, \psi \right).
	\end{align*}
	
	On the other hand, subtracting $H_{\phi^{n}}\phi^{n}$ from \eqref{eq_b:GF-semi-discretization} gives
	\begin{align*}
		\mu^{n+1} - H_{\phi^{n}}\phi^{n} = -\frac{1}{2}\Delta \left( \phi^{n+1} - \phi^{n} \right) + \alpha \left( \phi^{n+1} - \phi^{n} \right) = \frac{\tau}{2} \Delta \mu^{n+1} - \tau \alpha \mu^{n+1}.
	\end{align*}
	Rearranging the terms above yields the following elliptic resolvent equation (with the same boundary conditions as in \eqref{eq:GF-semi-discretization}):
	\begin{align}
		(1 + \tau\alpha) \mu^{n+1} - \frac{\tau}{2} \Delta \mu^{n+1} = H_{\phi^{n}}\phi^{n}. \label{eq:1-ellipse-equation-mu}
	\end{align}
	We find from \eqref{eq:1-ellipse-equation-mu} and the the action decay \eqref{eq:S_decaying} that for any $\psi \in X$, 
	\begin{align*}
		\left\langle H_{\phi^{n}} \phi^{n}, \psi \right\rangle =&\, (1 + \tau\alpha) \left\langle \mu^{n+1}, \psi \right\rangle + \frac{\tau}{2} \left\langle \nabla \mu^{n+1}, \nabla \psi \right\rangle \lesssim (1 + \tau) \left\| \mu^{n+1} \right\|_{H^{1}} \left\| \psi \right\|_{H^{1}} \\
		\lesssim&\, (1 + \tau^{-1}) \sqrt{ S_{\Omega, \omega}(\phi^{n}) - S_{\Omega, \omega}(\phi^{n+1}) } \left\| \psi \right\|_{H^{1}} \xrightarrow{n\to\infty} 0.
	\end{align*}
This means $\left| a\left( \phi^{n_{j}}, \psi \right) \right| = \left| \left\langle H_{\phi^{n_{j}}} \phi^{n_{j}}, \psi \right\rangle \right| \xrightarrow{j\to\infty} 0$, and so
	\begin{align*}
		a\left( \phi^{\star}, \psi \right) = \left\langle H_{\phi^{\star}} \phi^{\star}, \psi \right\rangle  = 0, \quad \forall\, \psi\in X,
	\end{align*}
	which confirms \eqref{eq:Euler-LagrangeEq-weakform}. Finally, since the action functional sequence $\{S_{\Omega, \omega}(\phi^n)\}$ is non-increasing and bounded from below, it has a unique limit, i.e.,
	\begin{align*}
		S_{\Omega, \omega}(\phi^{\star})=\lim_{j\to\infty}S_{\Omega, \omega}(\phi^{n_j})=\lim_{n\to\infty}S_{\Omega, \omega}(\phi^n)=\inf_{n\geq0}S_{\Omega, \omega}(\phi^n),
	\end{align*}
	which completes the proof.
\end{proof}

\begin{remark}
	For the case of homogeneous Dirichlet boundary conditions with some regularity assumption on the boundary $\partial\Omega$, we could remove the additional assumption $p\leq d/(d-2)$ for $d\geq 3$ in Theorem~\ref{thm:convergent-subsequences} by employing the $W^{2,q}$ estimate (see, e.g., \cite[Sect.~9.6]{GT}) for the elliptic equation \eqref{eq:I-rearrange}:
	\begin{align}\label{eq:Lq-estimate}
		\left\|\phi^{n+1}\right\|_{W^{2,q}}\leq C(\tau) \left( \left\| \phi^{n} \right\|_{L^{q}} + \tau \left\| g(\phi^{n}) \right\|_{L^{q}} \right), 
	\end{align}
	where $q=\frac{2d}{p(d-2)}\in \left(\frac{2d}{d+2},\frac{2d}{d-2}\right)$ and $C(\tau)>0$ depends on $\tau$ but is independent of $n$. In fact, applying H\"older's inequality and the continuous embeddings $H^1(D)\hookrightarrow L^{q}(D)$ and $H^1(D)\hookrightarrow L^{pq}(D)$ gives
    \begin{align*}
    	\left\|g(\phi^n)\right\|_{L^q}\lesssim \left\|\phi^{n}\right\|_{L^q} + \left\|\phi^{n}\right\|_{H^1}^q + \left\|\phi^{n}\right\|_{L^{pq}}^p \lesssim \left\|\phi^{n}\right\|_{H^1} + \left\|\phi^{n}\right\|_{H^1}^q +  \left\|\phi^{n}\right\|_{H^1}^p \lesssim 1,
    \end{align*}
    which, combining with \eqref{eq:Lq-estimate}, implies the boundedness of $\{\phi^n\}_{n=1}^{\infty}$ in $W^{2,q}(D)$ for fixed $\tau>0$, and the subsequential convergence in $H^1$ can be obtained by the compact embedding $W^{2,q}(D)\hookrightarrow H^1(D)$.
\end{remark}

\subsection{Local convergence rate}

This subsection is devoted to rigorously establishing a practically observed exponential convergence rate of the DGF scheme \eqref{eq:GF-semi-discretization}. To achieve this goal, let us first discuss the uniqueness of GS.

\subsubsection{Local uniqueness of GS}

For any $\phi \in X$, we have $\mathrm{e}^{\mathrm{i}\theta}\phi \in X$ and $S_{\Omega,\omega}(\mathrm{e}^{\mathrm{i}\theta}\phi) = S_{\Omega,\omega}(\phi)$. Consequently, the (local) uniqueness of the GS can only be expected up to a constant phase factor $\mathrm{e}^{\mathrm{i}\theta}$ with $\theta\in(-\pi,\pi]$. The action functional $S_{\Omega,\omega}$ may also be invariant under other symmetry transformations in special geometric settings, for instance, under rotations of the coordinate system when the domain $D$ is a ball and the potential $V$ is radially symmetric. Such additional symmetries are not considered in the present work, and let us focus on the phase invariance inherent to the complex-valued formulation.

Denote the phase orbit $\llbracket \phi \rrbracket := \left\{ \mathrm{e}^{\mathrm{i} \theta} \phi \,\colon\, \theta\in \mathbb{R} \right\}$ of $\phi\in X$, and define the phase-invariant $H^{1}$-distance between two orbits by
\begin{align*}
	\mathrm{dist}_{H^{1}}\left[ \phi_{1}, \phi_{2} \right] := \inf_{u_{1}\in \llbracket \phi_{1} \rrbracket, u_{2}\in \llbracket \phi_{2} \rrbracket }\left\| u_{1} - u_{2} \right\|_{H^{1}}, \quad \forall \, \phi_{1}, \phi_{2} \in X.
\end{align*}
We now show that this definition of the distance is well posed.

\begin{lemma}[Orbit distance]\label{lem:h1-distance}
	For any $\phi_{1},\phi_{2}, \phi_{3} \in X$, the following properties hold:
	\begin{enumerate}[label=(\roman*)]
		\item Attainment of the infimum: there exists $\theta \in(-\pi,\pi]$ such that
		\begin{align*}
			\mathrm{dist}_{H^{1}}\left[ \phi_{1}, \phi_{2} \right]
			= \inf_{u\in \llbracket \phi_{2} \rrbracket} \left\| \phi_{1} - u \right\|_{H^{1}}
			= \left\| \phi_{1} - \mathrm{e}^{\mathrm{i}\theta}\phi_{2} \right\|_{H^{1}}.
		\end{align*}
		
		\item Non-negativity and separation: $\mathrm{dist}_{H^{1}}\left[ \phi_{1}, \phi_{2} \right] \ge 0$, with equality if and only if $\phi_{1}\in \llbracket \phi_{2} \rrbracket$ (or, $\phi_{2}\in \llbracket \phi_{1} \rrbracket$).
		
		\item Symmetry: $\mathrm{dist}_{H^{1}}\left[ \phi_{1}, \phi_{2} \right] = \mathrm{dist}_{H^{1}}\left[ \phi_{2}, \phi_{1} \right]$.
		
		\item Triangle inequality: $\mathrm{dist}_{H^{1}}\left[ \phi_{1}, \phi_{2} \right] \le \mathrm{dist}_{H^{1}}\left[ \phi_{1}, \phi_{3} \right] + \mathrm{dist}_{H^{1}}\left[ \phi_{3}, \phi_{2} \right]$.
	\end{enumerate}
\end{lemma}
\begin{proof}
	First, by reducing the two phase parameters to one, we have
	\begin{align*}
		\mathrm{dist}_{H^{1}}\left[ \phi_{1}, \phi_{2} \right] =&\, \inf_{\theta_{1}, \theta_{2} \in \mathbb{R}} \left\| \mathrm{e}^{\mathrm{i} \theta_{1}} \phi_{1} - \mathrm{e}^{\mathrm{i} \theta_{2}} \phi_{2} \right\|_{H^{1}} = \inf_{\theta_{1}, \theta_{2} \in \mathbb{R}} \left\| \phi_{1} - \mathrm{e}^{\mathrm{i} \left( \theta_{2} - \theta_{1} \right)}\phi_{2} \right\|_{H^{1}} \\
		=&\, \inf_{\theta \in \mathbb{R}} \left\| \phi_{1} - \mathrm{e}^{\mathrm{i} \theta}\phi_{2} \right\|_{H^{1}} = \inf_{u \in \llbracket \phi_{2} \rrbracket} \left\| \phi_{1} - u \right\|_{H^{1}}.
	\end{align*}
	Since the function $\theta \mapsto \| \phi_{1} - \mathrm{e}^{i\theta}\phi_{2} \|_{L^{2}}^2$ is a smooth $2\pi$-periodic function, it attains at least one minimum on $(-\pi,\pi]$, which proves \textit{(i)}. Assertions \textit{(ii)}–\textit{(iii)} are immediate. To prove $(iv)$, let $\theta_{1},\theta_{2}\in\mathbb{R}$ satisfy
	\begin{align*}
		\mathrm{dist}_{H^{1}}\left[ \phi_{1}, \phi_{3} \right] = \left\| \phi_{1} - \mathrm{e}^{\mathrm{i} \theta_{1}}\phi_{3} \right\|_{H^{1}}, \quad \mathrm{dist}_{H^{1}}\left[ \phi_{3}, \phi_{2} \right] = \left\| \phi_{3} - \mathrm{e}^{\mathrm{i} \theta_{2}} \phi_{2} \right\|_{H^{1}},
	\end{align*}
	and then
	\begin{align*}
		\mathrm{dist}_{H^{1}}\left[ \phi_{1}, \phi_{2} \right] \leq&\, \left\| \phi_{1} - \mathrm{e}^{\mathrm{i} \left( \theta_{1} + \theta_{2} \right)}\phi_{2} \right\|_{H^{1}} \leq \left\| \phi_{1} - \mathrm{e}^{\mathrm{i} \theta_{1}} \phi_{3} \right\|_{H^{1}} + \left\| \mathrm{e}^{\mathrm{i} \theta_{1}} \phi_{3} - \mathrm{e}^{\mathrm{i} \left( \theta_{1} + \theta_{2} \right)}\phi_{2} \right\|_{H^{1}} \\
		=&\, \left\| \phi_{1} - \mathrm{e}^{\mathrm{i} \theta_{1}}\phi_{3} \right\|_{H^{1}} + \left\| \phi_{3} - \mathrm{e}^{\mathrm{i} \theta_{2}} \phi_{2} \right\|_{H^{1}} = \mathrm{dist}_{H^{1}}\left[ \phi_{1}, \phi_{3} \right] + \mathrm{dist}_{H^{1}}\left[ \phi_{3}, \phi_{2} \right],
	\end{align*}
	which completes the proof.
\end{proof}

For any GS $\phi_{g}$ in \eqref{gs unconstrain def}, define the neighborhood $U_{\delta}(\phi_{g}) := \left\{ \phi\in X \,\colon\, \mathrm{dist}_{H^{1}}\left[ \phi, \phi_{g} \right] \leq \delta \right\}$. Note that for every $u=\mathrm{e}^{\mathrm{i} \theta}\phi_{g} \in \llbracket \phi_{g} \rrbracket$, the vector $\mathrm{i}u$ is tangent to the phase orbit at $u$ since $\frac{\mathrm d}{\mathrm d \theta}\mathrm{e}^{\mathrm{i} \theta}\phi_{g} = \mathrm{i} \mathrm{e}^{\mathrm{i} \theta}\phi_{g} = \mathrm{i}u$. Accordingly, we introduce the $H^1$-orthogonal complement of $\mathrm{i}u$:
\[ 
\left\{ \mathrm i u \right\}_{H^{1}}^{\perp} := \big\{ \eta \in X \,\colon\, \left\langle \mathrm{i}u, \eta \right\rangle_{H^{1}(D)} = 0 \big\}.
\]
For any $\phi \in U_{\delta}(\phi_{g})$, we define $\theta_{\phi}$ by minimizing the distance on $(-\pi,\pi]$:
\begin{align*}
	\theta_{\phi} \in \arg \min_{\theta \,\in\, (-\pi, \pi]} \left\| \phi - \mathrm{e}^{\mathrm{i} \theta} \phi_{g} \right\|_{H^{1}}.
\end{align*}

\begin{lemma}\label{prop:unique-theta}
	If $\delta > 0$ is sufficiently small, then for any $\phi \in U_{\delta}(\phi_{g})$, the minimizer $$\theta_{\phi} = \operatorname*{Arg}\left( \langle \phi, \phi_g \rangle + \langle \nabla\phi, \nabla\phi_g \rangle \right)$$ is unique, where $\operatorname{Arg}$ denotes the principal argument angle of a complex number. Moreover, setting $u_{\phi} = \mathrm{e}^{\mathrm{i} \theta_{\phi}} \phi_{g} \in \llbracket \phi_{g} \rrbracket$, we have $\phi - u_{\phi} \in \left\{ \mathrm i u_{\phi} \right\}_{H^{1}}^{\perp}$.
\end{lemma}
\begin{proof}
	Denote $F(\theta) := \frac12 \| \phi - \mathrm{e}^{\mathrm i\theta} \phi_g \|_{H^{1}}^2
	= \frac12 \|\phi\|_{H^{1}}^2 + \frac12 \|\phi_g\|_{H^{1}}^2 - \langle \phi, \mathrm{e}^{\mathrm i\theta} \phi_g \rangle_{H^{1}(D)}$, and its derivative reads
	\begin{align*}
		F'(\theta) = -\langle \phi, \mathrm i \mathrm{e}^{\mathrm i\theta} \phi_g \rangle_{H^{1}(D)}
		= -\operatorname{Im} \left( \langle \phi, \mathrm{e}^{\mathrm i\theta} \phi_g \rangle + \langle \nabla\phi, \mathrm{e}^{\mathrm i\theta} \nabla\phi_g \rangle \right).
	\end{align*}
	Denoting $\langle \phi, \phi_g \rangle + \langle \nabla\phi, \nabla\phi_g \rangle = r_{\phi} \, \mathrm{e}^{\mathrm{i} \theta_{\phi}}$ with
	\begin{align*}
		r_{\phi} := \left| \langle \phi, \phi_g \rangle + \langle \nabla\phi, \nabla\phi_g \rangle \right|, \quad \theta_{\phi} := \operatorname*{Arg}\left( \langle \phi, \phi_g \rangle + \langle \nabla\phi, \nabla\phi_g \rangle \right),
	\end{align*} 
	we have
	\begin{align*}
		\langle \phi, \mathrm{e}^{\mathrm i\theta} \phi_g \rangle + \langle \nabla\phi, \mathrm{e}^{\mathrm i\theta} \nabla\phi_g \rangle = r_{\phi} \mathrm{e}^{\mathrm{i} \left( \theta_{\phi} - \theta \right)}.
	\end{align*}
	Hence,
	\begin{align*}
		F'(\theta) = -r_{\phi} \sin\left( \theta_{\phi} - \theta \right), \quad F''(\theta) = r_{\phi} \cos\left( \theta_{\phi} - \theta \right).
	\end{align*}
	The stationary condition $F'(\theta) = 0$ yields $\theta = \theta_{\phi} - k\pi=\theta_{k}$, $k\in \mathbb{Z}$. At a stationary point $\theta_{k}$, there is $F''(\theta_{k}) = (-1)^{k}r_{\phi}$. Hence, for even $k$ and enough small $\delta$, $F''(\theta_{k}) = r_{\phi} > 0$, showing a local minimum. This also means that there is only one minimizer $\theta_{\phi}$ in $(-\pi,\pi]$, giving the uniqueness. Finally,
	\begin{align*}
		\left\langle \mathrm{i}u_{\phi}, \phi - u_{\phi} \right\rangle_{H^{1}(D)} =&\, \left\langle \mathrm{i} \mathrm{e}^{\mathrm{i} \theta_{\phi}}\phi_{g}, \phi - \mathrm{e}^{\mathrm{i} \theta_{\phi}}\phi_{g} \right\rangle_{H^{1}(D)} = \left\langle \mathrm{i} \mathrm{e}^{\mathrm{i} \theta_{\phi}}\phi_{g}, \phi \right\rangle_{H^{1}(D)} - \left\langle \mathrm{i} \mathrm{e}^{\mathrm{i} \theta_{\phi}}\phi_{g}, \mathrm{e}^{\mathrm{i} \theta_{\phi}}\phi_{g} \right\rangle_{H^{1}(D)} \\
		=&\, \left\langle \mathrm{i} \mathrm{e}^{\mathrm{i} \theta_{\phi}}\phi_{g}, \phi \right\rangle_{H^{1}(D)} = - F'(\theta_{\phi}) = 0,
	\end{align*}
	so $\phi - u_{\phi} \in \{\mathrm i u_{\phi}\}_{H^{1}}^{\perp}$. This completes the proof.
\end{proof}

With Lemma~\ref{prop:unique-theta}, for a sufficiently small $\delta > 0$, we can have the decomposition
\begin{align*}
	\phi = u_{\phi} + \left( \phi - u_{\phi} \right) = \mathrm{e}^{\mathrm{i} \theta_{\phi}} \left[ \phi_{g} + \left( \mathrm{e}^{-\mathrm{i} \theta_{\phi}}\phi - \phi_{g} \right) \right],
\end{align*}
and note that $\mathrm{e}^{-\mathrm{i}\theta_{\phi}}\phi - \phi_{g} \in \{\mathrm{i}\phi_{g}\}_{H^{1}}^{\perp}$. Hence, we can introduce a local coordinate mapping
\begin{align*} 
	\chi: (-\pi, \pi]\times \{\mathrm i \phi_{g}\}_{H^{1}}^{\perp} \mapsto U_{\delta}(\phi_{g}),\quad (\theta, v)\mapsto \mathrm{e}^{\mathrm{i} \theta}\left( \phi_{g} + v \right),
\end{align*}
which is a bijection onto its image. Its inverse is given by 
\begin{align*}
	\chi^{-1}:\ \phi\in U_{\delta}(\phi_{g}) \mapsto 
	\bigl(\theta(\phi), v(\phi)\bigr)
	=\left( \theta_{\phi}, \mathrm{e}^{-\mathrm{i} \theta_{\phi}}\phi - \phi_{g} \right)\in (-\pi, \pi]\times \{\mathrm i \phi_{g}\}_{H^{1}}^{\perp}.
\end{align*}
For convenience, set 
\begin{align*}
	v_{\phi} := \mathrm{e}^{-\mathrm{i}\theta_{\phi}}\phi - \phi_{g}, \quad w_{\phi} := \phi - u_{\phi} = \mathrm{e}^{\mathrm{i} \theta_{\phi}} v_{\phi},
\end{align*}
so that $\mathrm{dist}_{H^{1}}\!\left[ \phi, \phi_{g} \right] = \| w_{\phi} \|_{H^{1}} = \| v_{\phi} \|_{H^{1}}$. For any $u, \eta_{1}, \eta_{2} \in X$, the first- and second-order G\^{a}teaux derivatives of the action functional read
\begin{align*}
	\mathcal{D}S_{\Omega,\omega}(u)[\eta_{1}] =&\, 2\left\langle \left( -\frac{1}{2}\Delta + V +\beta|u|^{p-1} - \Omega L_{z} + \omega \right)\phi, \eta_{1} \right\rangle_{L^{2}(D)}, \\
	\mathcal{D}^{2}S_{\Omega,\omega}(u)[\eta_{2},\eta_{1}] =&\, 2\left\langle \left( -\frac{1}{2}\Delta + V - \Omega L_{z} + \omega \right)\eta_{2}, \eta_{1} \right\rangle_{L^{2}(D)} \\
	&\, + 2\left\langle \frac{p-1}{2}\beta |u|^{p-3} u^{2}\overline{\eta_{2}} + \frac{p+1}{2}\beta |u|^{p-1}\eta_{2}, \eta_{1} \right\rangle_{L^{2}(D)}.
\end{align*}
Define
\begin{align*}
	\mathcal{L}_{1} := -\frac{1}{2}\Delta + V - \Omega L_{z} + \omega, \quad \mathcal{L}_{2,u} \eta := \frac{p-1}{2}\beta |u|^{p-3} u^{2}\overline{\eta} + \frac{p+1}{2}\beta |u|^{p-1}\eta, \quad \mathcal{L}_{u} := \mathcal{L}_{1} + \mathcal{L}_{2,u}.
\end{align*}
Then $\mathcal{D}^{2}S_{\Omega,\omega}(u)[\eta_{2},\eta_{1}] = 2\left\langle \mathcal{L}_{u} \eta_{2}, \eta_{1} \right\rangle_{L^{2}(D)}$. Since only $p>1$ is assumed, higher-order variations may not be available. Hence, we require continuity of the second-order G\^{a}teaux derivative with respect to $u$. Recall the following result.

\begin{lemma}[Theorem A.2 of \cite{willem1996minimax}] \label{lem:f2}
	On a bounded domain $D$, for $1<p<\infty$, the mappings $u\mapsto|u|^{p-1}$ and $u\mapsto |u|^{p-3}u^{2}$ from $L^{p+1}(D)$ to $L^{\frac{p+1}{p-1}}(D)$ are continuous.
\end{lemma}
For any $\rho\in[0,1]$, using H\"older's inequality and Sobolev's embeddings and Lemma~\ref{lem:f2}, we have
\begin{align*}
	&\,\left\langle \left( \mathcal{L}_{2,u_{\phi} + \rho w_{\phi}} - \mathcal{L}_{2,u_{\phi}} \right) w_{\phi}, w_{\phi} \right\rangle_{L^{2}(D)} \\
	=&\, \frac{(p-1)\beta}{2}\left\langle \left( |u_{\phi} + \rho w_{\phi}|^{p-3}(u_{\phi} + \rho w_{\phi})^{2} - |u_{\phi}|^{p-3}(u_{\phi})^{2} \right) \overline{w_{\phi}}, w_{\phi} \right\rangle_{L^{2}(D)} \\
	&\, +  \frac{(p+1)\beta}{2}\left\langle \left( |u_{\phi} + \rho w_{\phi}|^{p-1} - |u_{\phi}|^{p-1} \right) w_{\phi}, w_{\phi} \right\rangle_{L^{2}(D)} \\
	\lesssim&\, \left( \left\| |u_{\phi} + \rho w_{\phi}|^{p-3}(u_{\phi} + \rho w_{\phi})^{2} - |u_{\phi}|^{p-3}(u_{\phi})^{2} \right\|_{L^{\frac{p+1}{p-1}}} + \left\| |u_{\phi} + \rho w_{\phi}|^{p-1} - |u_{\phi}|^{p-1} \right\|_{L^{\frac{p+1}{p-1}}} \right) \left\| w_{\phi} \right\|_{L^{p+1}}^{2}.
\end{align*}
Consequently,
\begin{align}
	\left\langle \left( \mathcal{L}_{2,u_{\phi} + \rho w_{\phi}} - \mathcal{L}_{2,u_{\phi}} \right) w_{\phi}, w_{\phi} \right\rangle_{L^{2}(D)} = o\left( \left\| w_{\phi} \right\|_{L^{p+1}}^{2} \right) = o\left( \left\| w_{\phi} \right\|_{H^{1}}^{2} \right), \label{eq:L-L}
\end{align}
which will be used frequently below. For subsequent technical analysis, we require $\phi_{g}$ (and all its phase shifts) to be a \textbf{non‑degenerate} minimizer of $S_{\Omega,\omega}$ on $X$ by imposing a standard coercivity condition on the second variation of $S_{\Omega,\omega}$.

\begin{proposition}[Local uniqueness]\label{prop:coercive-second-variation}
	Assume that there exists a constant $c>0$ such that
	\begin{align}
		\mathcal{D}^{2}S_{\Omega,\omega} (\phi_{g}) [\eta,\eta] \ge c\left\| \eta \right\|_{H^{1}}^{2}, \quad \forall \, \eta\in \{\mathrm i \phi_{g}\}_{H^{1}}^{\perp}. \label{eq:coercivity-second-variation}
	\end{align}
	Then, for a sufficiently small $\delta>0$, we have
	\begin{align}
		S_{\Omega,\omega}(\phi) - S_{\Omega,\omega}(\phi_{g}) \asymp \left( \mathrm{dist}_{H^{1}}\left[ \phi, \phi_{g} \right] \right)^{2}, \label{eq:energy-gap-quadratic}
	\end{align}
	and the unconstrained minimization problem \eqref{gs unconstrain def} has a minimizer in $U_{\delta}(\phi_{g})$ if and only if it is $\phi_{g}$ or its phase shifts.
\end{proposition}
\begin{proof}
The condition \eqref{eq:coercivity-second-variation} gives $\left\langle \mathcal{L}_{\phi_{g}} \eta, \eta \right\rangle_{L^{2}(D)} \ge \frac{c}{2}\left\| \eta \right\|_{H^{1}}^{2}$, and, moreover,
	\begin{align*}
		\left| \left\langle \mathcal{L}_{\phi_{g}}\eta, \eta \right\rangle \right| \lesssim \left\| \eta \right\|_{H^{1}}^{2} + \beta\, p\, \left\| \phi_{g} \right\|_{L^{p+1}}^{p-1} \left\| \eta \right\|_{L^{p+1}}^{2} \lesssim \left\| \eta \right\|_{H^{1}}^{2},
	\end{align*}
	hence $\left\langle \mathcal{L}_{\phi_{g}}\eta, \eta \right\rangle_{L^{2}(D)} \asymp \left\| \eta \right\|_{H^{1}}^{2}$ for all $\eta\in \{\mathrm i \phi_{g}\}_{H^{1}}^{\perp}$. According to the Taylor expansion, there is
	\begin{align*}
		&S_{\Omega,\omega}(\phi) - S_{\Omega,\omega}(\phi_{g}) = S_{\Omega,\omega}(\phi) - S_{\Omega,\omega}(u_{\phi}) \\
		=&\, \mathcal{D}S_{\Omega,\omega}(u_{\phi})[w_{\phi}] + \int_{0}^{1}{ (1-\rho)\mathcal{D}^{2}S_{\Omega,\omega}(u_{\phi} + \rho w_{\phi})[w_{\phi}, w_{\phi}] }\, \mathrm{d}\rho \\
		=&\,\frac{1}{2}\mathcal{D}^{2}S_{\Omega,\omega}(u_{\phi})[w_{\phi}, w_{\phi}]  + \int_{0}^{1}{ (1-\rho)\left( \mathcal{D}^{2}S_{\Omega,\omega}(u_{\phi} + \rho w_{\phi})[w_{\phi}, w_{\phi}] - \mathcal{D}^{2}S_{\Omega,\omega}(u_{\phi})[w_{\phi}, w_{\phi}] \right) }\, \mathrm{d}\rho \\
		=&\, \frac{1}{2}\mathcal{D}^{2}S_{\Omega,\omega}(u_{\phi})[w_{\phi}, w_{\phi}] + \int_{0}^{1}{ (1-\rho)\left\langle \left( \mathcal{L}_{2, u_{\phi} + \rho w_{\phi}} - \mathcal{L}_{2, u_{\phi}} \right) w_{\phi}, w_{\phi} \right\rangle_{L^{2}(D)} }\, \mathrm{d}\rho \\
		=&\, \frac{1}{2}\mathcal{D}^{2}S_{\Omega,\omega}(u_{\phi})[w_{\phi}, w_{\phi}] + o\left( \left\| w_{\phi} \right\|_{H^{1}}^{2} \right), 
	\end{align*}
	where we used \eqref{eq:L-L} and the first-order critical condition $\mathcal{D}S_{\Omega,\omega}(u_{\phi})[w_{\phi}]=0$. Noting that
	\begin{align*}
		\mathcal{D}^{2}S_{\Omega,\omega}(u_{\phi})[w_{\phi}, w_{\phi}] = \mathcal{D}^{2}S_{\Omega,\omega}(\mathrm{e}^{\mathrm{i} \theta_{\phi}}\phi_{g})[\mathrm{e}^{\mathrm{i} \theta_{\phi}} v_{\phi}, \mathrm{e}^{\mathrm{i} \theta_{\phi}}v_{\phi}] = \mathcal{D}^{2}S_{\Omega,\omega}(\phi_{g})[ v_{\phi}, v_{\phi}],
	\end{align*}
	we conclude that for all $u\in U_{\delta}(\phi_{g})$ with  $\delta>0$ small enough,
	\begin{align*}
		S_{\Omega,\omega}(\phi) - S_{\Omega,\omega}(\phi_{g}) =&\, \frac{1}{2} \mathcal{D}^{2}S_{\Omega,\omega}(\phi_{g})[ v_{\phi}, v_{\phi}] + o\left( \left\| v_{\phi} \right\|_{H^{1}}^{2} \right) \\
        =&\, \left\langle \mathcal{L}_{\phi_{g}}v_{\phi}, v_{\phi} \right\rangle_{L^{2}(D)} + o\left( \left\| v_{\phi} \right\|_{H^{1}}^{2} \right) \\
		\asymp&\, \left\| v_{\phi} \right\|_{H^{1}}^{2} = \left( \mathrm{dist}_{H^{1}}\left[ \phi, \phi_{g} \right] \right)^{2}.
	\end{align*}
	This proves the proposition.
\end{proof}

\begin{remark}
The non-degeneracy assumption can also be characterized via the $L^2$-orthogonal complement (similar to, e.g., \cite{henning2025discrete}):
\[ 
\left\{ \mathrm i \phi_{g} \right\}_{L^{2}}^{\perp} := \big\{ \xi \in X \,\colon\, \left\langle \mathrm{i}\phi_{g}, \xi \right\rangle_{L^{2}(D)} = 0 \big\}. 
\]
In that case, the condition \eqref{eq:coercivity-second-variation} is replaced by
    \begin{align}
		\mathcal{D}^{2}S_{\Omega,\omega} (\phi_{g}) [\xi,\xi] \ge c\left\| \xi \right\|_{H^{1}}^{2}, \quad \forall \, \xi\in \{\mathrm i \phi_{g}\}_{L^{2}}^{\perp}. \label{eq:coercivity-second-variation-L2}
	\end{align}
    Here we show that \eqref{eq:coercivity-second-variation-L2} is the stronger requirement, as it implies \eqref{eq:coercivity-second-variation}. In fact, for any $\eta \in \left\{ \mathrm i \phi_{g} \right\}_{H^{1}}^{\perp}$, we have $L^{2}$-orthogonal decomposition as
    \begin{align*}
    	\eta = \mathrm i \kappa \phi_{g} + \xi 
        \quad\;\mbox{with}\;\; 
    	\kappa = \frac{\left\langle  \mathrm{i}\phi_{g}, \eta \right\rangle_{L^{2}(D)}}{\left\| \phi_{g} \right\|_{L^{2}}^{2}} \in \mathbb{R}
        \;\;\mbox{and}\;\;
        \xi = \eta - \frac{\mathrm i \left\langle  \mathrm{i}\phi_{g}, \eta \right\rangle_{L^{2}(D)} \phi_{g}}{\left\| \phi_{g} \right\|_{L^{2}}^{2}} \in \left\{ \mathrm i \phi_{g} \right\}_{L^{2}}^{\perp}.
    \end{align*}
    Substituting this decomposition into $\mathcal{D}^{2}S_{\Omega,\omega}(\phi_{g})[\eta,\eta]$ yields
    \begin{align*}
    	\mathcal{D}^{2}S_{\Omega,\omega}(\phi_{g})[\eta,\eta] =&\, \kappa^{2} \mathcal{D}^{2}S_{\Omega,\omega}(\phi_{g})[\mathrm{i}\phi_{g},\mathrm{i}\phi_{g}] + 2 \kappa \mathcal{D}^{2}S_{\Omega,\omega}(\phi_{g})[\mathrm{i}\phi_{g}, \xi] + \mathcal{D}^{2}S_{\Omega,\omega}(\phi_{g})[\xi, \xi] \\
        =&\, 2 \kappa^{2} \left\langle \mathrm{i} H_{\phi_{g}}\phi_{g}, \mathrm{i}\phi_{g} \right\rangle_{L^{2}(D)} + 4 \kappa \left\langle \mathrm{i} H_{\phi_{g}}\phi_{g}, \xi \right\rangle_{L^{2}(D)} + \mathcal{D}^{2}S_{\Omega,\omega}(\phi_{g})[\xi, \xi] \\
        =&\, \mathcal{D}^{2}S_{\Omega,\omega}(\phi_{g})[\xi, \xi].
    \end{align*}
    On the other hand, 
    \begin{align*}
    	\left\| \eta \right\|_{H^{1}}^{2} =&\, \kappa^{2} \left\| \phi_{g} \right\|_{H^{1}}^{2} + 2\kappa \left\langle \mathrm{i}\phi_{g}, \xi \right\rangle_{H^{1}(D)} + \left\| \xi \right\|_{H^{1}}^{2} \\
        =&\, \kappa^{2} \left\| \phi_{g} \right\|_{H^{1}}^{2} + 2\kappa \left\langle \mathrm{i}\phi_{g}, \eta - \mathrm{i}\kappa \phi_{g} \right\rangle_{H^{1}(D)} + \left\| \xi \right\|_{H^{1}}^{2} \\
        =&\, -\kappa^{2} \left\| \phi_{g} \right\|_{H^{1}}^{2} + \left\| \xi \right\|_{H^{1}}^{2},
    \end{align*}
    which means $\left\| \eta \right\|_{H^{1}} \le \left\| \xi \right\|_{H^{1}}$. Therefore, by \eqref{eq:coercivity-second-variation-L2},
    \begin{align*}
    	\mathcal{D}^{2}S_{\Omega,\omega}(\phi_{g})[\eta,\eta] = \mathcal{D}^{2}S_{\Omega,\omega}(\phi_{g})[\xi, \xi] \ge c\left\| \xi \right\|_{H^{1}}^{2} \ge c\left\| \eta \right\|_{H^{1}}^{2}, \quad \forall \eta \in \left\{ \mathrm i \phi_{g} \right\}_{H^{1}}^{\perp},
    \end{align*}
    which is precisely \eqref{eq:coercivity-second-variation}.
\end{remark}

\subsubsection{Exponential convergence}

Now we are ready to present the convergence rate result for the DGF scheme \eqref{eq:GF-semi-discretization}. With the numerical solution $\phi^{n}$, we denote $\theta_{n} = \theta_{\phi^{n}}$, $u^{n} = u_{\phi^{n}}$, $v^{n} = v_{\phi^{n}}$ and $w^{n} = w_{\phi^{n}}$ for simplicity of notation. 

\begin{theorem}[Exponential rate]\label{thm:exponential-convergence-ground-state}
	Suppose that Assumption~\ref{assump:potential-condition} holds and $\omega < -\lambda_{0}$. Let $\phi_{g}$ be the GS defined by \eqref{gs unconstrain def} and suppose that it satisfies the non-degeneracy condition in Proposition~\ref{prop:coercive-second-variation}. Then, there exists a constant $\delta_{0}>0$ such that for any initial data $\phi^{0} \in U_{\delta_{0}}(\phi_{g})$, $\alpha \ge \alpha_{M} + \frac{1}{8}$ and $\tau>0$, the DGF scheme \eqref{eq:GF-semi-discretization} converges to the GS at the rate:
	\begin{align}\label{eq:exponential-convergence-to-ground-state}
		\mathrm{dist}_{H^{1}}\left[ \phi^{n}, \phi_{g} \right] \le C\,\mathrm{e}^{-a n \tau/(1+\tau)}, \qquad \forall\, n \in \mathbb{N},
	\end{align}
	where $C, a > 0$ are constants independent of $n$ and $\tau$.
\end{theorem}

We then turn to the proof of Theorem \ref{thm:exponential-convergence-ground-state}. As a preparation, we give some estimations derived from the action decaying property. We can show that the numerical solution $ \phi^{n} $ for all $n\in \mathbb{N}$ can stay within a local range of $U_{\delta}(\phi_{g})$, which is crucial.

\begin{lemma}[Lyapunov stability]\label{lem:local-stability-ground-state}
	Let the assumptions in Theorem~\ref{thm:exponential-convergence-ground-state} hold and let $\delta>0$ be a sufficiently small constant that meets the requirements of Lemma~\ref{prop:unique-theta} and Proposition~\ref{prop:coercive-second-variation}. Then, there exists a  $\delta_{0}>0$ such that for all $\phi^{0} \in U_{\delta_{0}}(\phi_{g})$, $\alpha \ge \alpha_{M} + \frac{1}{8}$ and $\tau>0$, the solution of \eqref{eq:GF-semi-discretization} with initial guess $\phi^{0}$ satisfies $\phi^{n}\in U_{\delta}(\phi_{g})$.
\end{lemma}
\begin{proof}
	From \eqref{eq:energy-gap-quadratic}, there exist constants $0<c_{1}\leq c_{2}$ such that
	\begin{align}\label{eq:energy-gap-quadratic-c1c2}
		c_1 \, \left( \mathrm{dist}_{H^{1}}\left[ \phi, \phi_{g} \right] \right)^{2} \leq S_{\Omega, \omega}(\phi) - S_{\Omega, \omega}(\phi_{g}) \leq c_2 \, \left( \mathrm{dist}_{H^{1}}\left[ \phi, \phi_{g} \right] \right)^{2}.
	\end{align}
	From the decay of action functional \eqref{eq:S_decaying}, we have
	\begin{align*}
		\tau^{2} \left\| \mu^{n+1} \right\|_{H^{1}}^{2} \leq 4\left( S_{\Omega, \omega}(\phi^{n}) - S_{\Omega, \omega}(\phi^{n+1}) \right), \quad 
	\end{align*}
	which means
	\begin{align*}
		\mathrm{dist}_{H^{1}}\left[ \phi^{n+1}, \phi^{n} \right] =&\, \inf_{\theta \in \mathbb{R}} \left\| \phi^{n+1} - \mathrm{e}^{\mathrm{i} \theta} \phi^{n} \right\|_{H^{1}} \leq \left\| \phi^{n+1} - \phi^{n} \right\|_{H^{1}} = \tau\left\| \mu^{n+1} \right\|_{H^{1}}\\ 
		\lesssim&\,  \sqrt{ S_{\Omega, \omega}(\phi^{n}) - S_{\Omega, \omega}(\phi^{n+1}) } \lesssim  \sqrt{ S_{\Omega, \omega}(\phi^{0}) - S_{\Omega, \omega}(\phi_{g}) } \lesssim \mathrm{dist}_{H^{1}}\left[ \phi^{0}, \phi_{g} \right], \;\;\; \forall n\in\mathbb{N}.
	\end{align*}
	Therefore, we can say for a constant $c_{3}>0$ that
	\begin{align}\label{eq:un-step-less-u0ug}
		\mathrm{dist}_{H^{1}}\left[ \phi^{n+1}, \phi^{n} \right] \leq c_{3} \, \mathrm{dist}_{H^{1}}\left[ \phi^{0}, \phi_{g} \right] \leq c_{3}\delta_{0}.
	\end{align}
	Fix $\delta_{0}>0$ such that $\delta_{0}<\frac{\delta}{2}\min\left\{ \sqrt{\frac{c_1}{c_2}}, \frac{1}{c_3}\right\}$. We now proceed to show that the assertion of the lemma holds for this $\delta_{0}$. We argue by contradiction, supposing that there exists $n_{1}\in \mathbb{N}^{+}$ with $\mathrm{dist}_{H^{1}}\left[ \phi^{n_{1}}, \phi_{g} \right] > \delta$. By \eqref{eq:un-step-less-u0ug}, $\mathrm{dist}_{H^{1}}\left[ \phi^{n+1}, \phi^{n} \right] \leq c_{3}\delta_{0}<\delta/2$. Combining with the triangle inequality of $\mathrm{dist}_{H^{1}}$ in Lemma~\ref{lem:h1-distance}, there exists $0 \neq n_{2} < n_{1}$ such that $\delta/2 <\mathrm{dist}_{H^{1}}\left[ \phi^{n_{2}}, \phi_{g} \right] \leq \delta$. Setting $\phi$ in \eqref{eq:energy-gap-quadratic-c1c2} as $\phi^{0}$ and $\phi^{n_{2}}$, respectively, one obtains
	\begin{align*}
		S_{\Omega, \omega}(\phi^{n_{2}}) -S_{\Omega, \omega}(\phi_{g}) &\geq  c_1 \left( \mathrm{dist}_{H^{1}}\left[ \phi^{n_{2}}, \phi_{g} \right] \right)^{2} 
		> c_1\delta^{2}/4, \\
		S_{\Omega, \omega}(\phi^{0}) -S_{\Omega, \omega}(\phi_{g}) &\leq c_2 \left( \mathrm{dist}_{H^{1}}\left[ \phi^{0}, \phi_{g} \right] \right)^{2} \leq c_2\delta_{0}^{2}.
	\end{align*}
	By choosing $\delta_{0}$, we can have $c_2\delta_0^2<c_1\delta^2/4$ and so $S_{\Omega, \omega}(\phi^{0}) < S_{\Omega, \omega}(\phi^{n_{2}})$, which contradicts the decay of action in Lemma~\ref{thm:action_decaying}.
\end{proof}

Under the non-degeneracy condition, we can establish the following \L{}ojasiewicz-type gradient inequality near the GS as the last tool.

\begin{lemma}[\L{}ojasiewicz-type gradient inequality]\label{lem:Lojasiewicz-Inequality}
	Suppose that $\omega < -\lambda_{0}$, and that Assumption~\ref{assump:potential-condition} as well as the non-degeneracy condition in Proposition~\ref{prop:coercive-second-variation} hold for the ground state $\phi_{g}$. Then, there exists a sufficiently small $\delta>0$ such that
	\begin{align}
		S_{\Omega, \omega}(\phi) - S_{\Omega, \omega}(\phi_{g}) \lesssim \left\| H_{\phi}\phi \right\|_{H^{-1}}^{2},
		\quad \forall\, \phi\in U_{\delta}(\phi_{g}). \label{eq:Lojasiewicz-Inequality}
	\end{align}
\end{lemma}
\begin{proof}
	For every $\phi \in U_{\delta}(\phi_{g})$, we have 
	\begin{align*}
		\left\langle H_{\phi}\phi, w_{\phi} \right\rangle_{L^{2}(D)} =  \left\langle H_{\phi}\phi - H_{u_{\phi}}u_{\phi}, w_{\phi} \right\rangle_{L^{2}(D)} = \left\langle \mathcal{L}_{1}w_{\phi}, w_{\phi} \right\rangle_{L^{2}(D)} + \beta\left\langle |\phi|^{p-1}\phi - |u_{\phi}|^{p-1}u_{\phi}, w_{\phi} \right\rangle_{L^{2}(D)}.
	\end{align*}
	Noting that
	\begin{align}
		&\left\langle |\phi|^{p-1}\phi - |u_{\phi}|^{p-1}u_{\phi}, w_{\phi} \right\rangle_{L^{2}(D)}  
        = \int_{0}^{1}{ \left\langle \mathcal{L}_{2, u_{\phi} + \rho w_{\phi}} w_{\phi}, w_{\phi} \right\rangle_{L^{2}(D)} }\, \mathrm{d}\rho \nonumber\\
		&\qquad = \left\langle \frac{p-1}{2} |u_{\phi}|^{p-3} u_{\phi}^{2}\overline{w_{\phi}} + \frac{p+1}{2} |u_{\phi}|^{p-1}w_{\phi}, w_{\phi} \right\rangle_{L^{2}(D)}  \!+ \int_{0}^{1}{ \left\langle \left( \mathcal{L}_{2, u_{\phi} + \rho w_{\phi}} - \mathcal{L}_{2, u_{\phi}} \right) w_{\phi}, w_{\phi} \right\rangle_{L^{2}(D)} }\, \mathrm{d}\rho \nonumber\\
		&\qquad =\left\langle \frac{p-1}{2} |\phi_{g}|^{p-3} \phi_{g}^{2}\overline{v_{\phi}} + \frac{p+1}{2} |\phi_{g}|^{p-1}v_{\phi}, v_{\phi} \right\rangle_{L^{2}(D)} + o\left( \left\| w_{\phi} \right\|_{H^{1}}^{2} \right), \label{eq:high-order-esti-2}
	\end{align}
	we obtain
	\begin{align*}
		\left\langle H_{\phi}\phi, w_{\phi} \right\rangle_{L^{2}(D)} = \left\langle \mathcal{L}_{\phi_{g}}v_{\phi}, v_{\phi} \right\rangle_{L^{2}(D)} + o\,( \left\| w_{\phi} \right\|_{H^{1}}^{2} ).
	\end{align*}
	From the definition of $H^{-1}$-norm, we have
	\begin{align*}
		\left\| H_{\phi}\phi \right\|_{H^{-1}} \geq \frac{\left\langle \mathcal{L}_{\phi_{g}}v_{\phi}, v_{\phi} \right\rangle}{\left\| w_{\phi} \right\|_{H^{1}}} + o\,( \left\| w_{\phi} \right\|_{H^{1}} ) \ge \frac{c}{2}\left\| v_{\phi} \right\|_{H^{1}} + o\,( \left\| v_{\phi} \right\|_{H^{1}} ).
	\end{align*}
	Using the local equivalence between $\|v_{\phi}\|_{H^{1}}^2$ and $S_{\Omega, \omega}(\phi) - S_{\Omega, \omega}(\phi_{g})$, we find when $\delta>0$ is small enough, 
	\begin{align*}
		\|H_{\phi}\phi\|_{H^{-1}}
		\ge \frac{c}{4}\|v_{\phi}\|_{H^{1}}
		\gtrsim \sqrt{S_{\Omega, \omega}(\phi) - S_{\Omega, \omega}(\phi_{g})},
	\end{align*}
	which proves the assertion.
\end{proof}

The following result is a direct application of Lemma~\ref{lem:local-stability-ground-state} and Lemma~\ref{lem:Lojasiewicz-Inequality}:

\begin{corollary}\label{cor:Lojasiewicz-Gradient-Inequality}
	Under the conditions of Theorem~\ref{thm:exponential-convergence-ground-state}, there exists some $\delta_{0}>0$ such that for all $\phi^{0}\in U_{\delta_{0}}(\phi_{g})$, $\alpha \ge \alpha_{M} + \frac{1}{8}$ and $\tau>0$,
	\begin{align}
		S_{\Omega, \omega}(\phi^{n}) - S_{\Omega, \omega}(\phi_{g}) \le C_{Lo}\,\|H_{\phi^{n}} \phi^{n}\|_{H^{-1}}^{2},
		\quad \forall\, n\in\mathbb{N}, \label{eq:Lojasiewicz-Gradient-Inequality}
	\end{align}
	with $C_{Lo}>0$ a constant independent of $n,\tau$.
\end{corollary}

Now we can present {the proof of Theorem \ref{thm:exponential-convergence-ground-state}} below.
\begin{proof}[\textbf{Proof of Theorem~\ref{thm:exponential-convergence-ground-state}}]
	From the elliptic equation \eqref{eq:1-ellipse-equation-mu}, we can  derive that
	\begin{align*}
		\left\langle H_{\phi^{n}} \phi^{n}, v \right\rangle =&\, (1 + \tau\alpha) \left\langle \mu^{n+1}, v \right\rangle + \frac{\tau}{2} \left\langle \nabla \mu^{n+1}, \nabla v \right\rangle \lesssim \left( \left\| \mu^{n+1} \right\|_{L^{2}} + \tau \left\| \mu^{n+1} \right\|_{H^{1}} \right) \left\| v \right\|_{H^{1}},
	\end{align*}
	which by noting the definition of $H^{-1}$-norm can give $ \|H_{\phi^{n}} \phi^{n}\|_{H^{-1}} \lesssim \left\|\mu^{n+1}\right\|_{L^{2}} + \tau \left\| \mu^{n+1} \right\|_{H^{1}}$. By Lemma~\ref{thm:action_decaying}, we have
	\begin{align*}
		\tau\|H_{\phi^{n}} \phi^{n}\|_{H^{-1}}^{2}
		&\lesssim \tau\|\mu^{n+1}\|_{L^{2}}^{2} + \tau^{3}\left\| \mu^{n+1} \right\|_{H^{1}}^{2} \\
        &\lesssim \left(1+\tau\right)\left(\tau\|\mu^{n+1}\|_{L^{2}}^{2} + \tau^{2}\left\| \mu^{n+1} \right\|_{H^{1}}^{2}\right) \\
		&\lesssim (1+\tau)\left(S_{\Omega, \omega}(\phi^{n}) - S_{\Omega, \omega}(\phi^{n+1})\right).
	\end{align*}
	Hence, for some constant $c_{D}>0$ satisfying $c_{D}\leq C_{Lo}$,
	\begin{align*}
		S_{\Omega, \omega}(\phi^{n}) - S_{\Omega, \omega}(\phi^{n+1}) \geq c_{D} \frac{\tau}{1+\tau} \|H_{\phi^{n}} \phi^{n}\|_{H^{-1}}^{2},
	\end{align*}
	which combining with \eqref{eq:Lojasiewicz-Gradient-Inequality} gives
	\begin{align*}
		S_{\Omega, \omega}(\phi^{n}) - S_{\Omega, \omega}(\phi^{n+1})
		\ge \frac{c_{D}}{C_{Lo}}\,\frac{\tau}{1+\tau}\big( S_{\Omega, \omega}(\phi^{n}) - S_{\Omega, \omega}(\phi_{g}) \big).
	\end{align*}
	Denoting $a_{0}=c_{D}/C_{Lo}\in(0,1]$, we then have 
	\begin{align}
		S_{\Omega, \omega}(\phi^{n+1}) - S_{\Omega, \omega}(\phi_{g}) \le \left(1-\frac{a_{0}\tau}{1+\tau}\right)\big( S_{\Omega, \omega}(\phi^{n}) - S_{\Omega, \omega}(\phi_{g}) \big),\quad n\in\mathbb{N}. \label{eq:Somega-nplus1-Somega-ug}
	\end{align}
	Using \eqref{eq:Somega-nplus1-Somega-ug} recursively, we find that for every $\tau>0$,
	\begin{align*}
		S_{\Omega, \omega}(\phi^{n}) - S_{\Omega, \omega}(\phi_{g})
		\le \left(1-\frac{a_{0}\tau}{1+\tau}\right)^{n}\big( S_{\Omega, \omega}(\phi^{0}) - S_{\Omega, \omega}(\phi_{g}) \big) 
		\le \mathrm{e}^{-na_0\tau/(1+\tau)}\big( S_{\Omega, \omega}(\phi^{0}) - S_{\Omega, \omega}(\phi_{g}) \big).
	\end{align*}
	Using the local equivalence $S_{\Omega, \omega}(\phi^{n})-S_{\Omega, \omega}(\phi_{g}) \asymp \left( \mathrm{dist}_{H^{1}}\left[ \phi^{n}, \phi_{g} \right] \right)^{2}$ from Proposition~\ref{prop:coercive-second-variation}, we complete the proof of Theorem~\ref{thm:exponential-convergence-ground-state} by setting $a=a_{0}/2$.
\end{proof}

\begin{remark}
   The semi-discrete DGF scheme \eqref{eq:GF-semi-discretization} can be coupled with a variety of spatial discretization methods, e.g., finite element and pseudospectral approaches, depending on the user's preference. The proposed theoretical analysis framework remains applicable to the fully discrete setting, provided that (i) a quantitative consistency estimate between the continuous optimization problem \eqref{eq:action-gs-def} and its discrete counterpart can be established (see similar efforts e.g., \cite{cances2010numerical,zhou2003analysis}), and (ii) the relevant inequalities, including the embedding and the Gagliardo–Nirenberg interpolation inequalities in Lemma~\ref{lem:standard-inequalities}, continue to hold under the spatial discretization. These issues will be addressed in another work. 
\end{remark}

\section{Numerical Verification} \label{sec. 4}

In this section, we give numerical examples to demonstrate the practical convergence behavior of the DGF scheme \eqref{eq:GF-semi-discretization}. We begin with a one-dimensional example in which analytical GS is available.

\begin{example}[Soliton in 1D]\label{ex:1}
	Letting $d=1$, $\mathbf{x}=x$, $V(x)=0$, $\beta=1$, $p=3$ and $D=[0,L]$ for a given $L>0$ in \eqref{eq:semilinear-elliptic-rot} yields
	\begin{align*}
		-\frac{1}{2}\partial_{xx}\phi(x) + |\phi(x)|^{2}\phi(x) = -\omega \phi(x), \quad x\in (0, L),
	\end{align*}
	with homogeneous Dirichlet boundary conditions. 
	Its action GS (unique positive solution) is the Jacobi elliptic function $\mathrm{sn}(\cdot, \cdot)$ (see \cite{carr2000stationary} for more details):
	\begin{align*}
		\phi_{g}(x) = \frac{2k \, K(k)}{L}\mathrm{sn} \left( \frac{2K(k)x}{L}, k \right)  , \quad  x\in [0, L],
	\end{align*}
	where $K(k) = \int_{0}^{\frac{\pi}{2}}{ \left( 1 - k^{2}\sin^{2} \theta \right)^{-1/2} }\, \mathrm{d}\theta$ is the complete elliptic integral of the first kind and the modulus $k\in [0,1]$ is determined by the equation $2(1+k^{2}) K(k)^{2} + \omega L^{2} = 0$. We take $L = 1$, and use the sine pseudospectral discretization with $N$ discrete sine modes. The stabilization factor can be chosen adaptively as 
	\begin{align*}
		\alpha^{n} = \frac{1}{2}\max \Big\{0,\ \max_{1\leq j \leq N-1}\left( V(x_{j})+\omega + \beta(p+2) |\phi^{n}(x_{j})|^{p-1} \right)\Big\},
	\end{align*}
	referring to \eqref{eq:alpha_n}, where $x_{j} = jL/N$, $j=0,1,\cdots, N$. The initial data is chosen as $\phi_{0}(x) = \sin(\pi x)$. We stop the iteration if the maximal residual is less than $10^{-13}$. i.e., 
	\begin{align*}
		\max_{1\leq j \leq N-1}\left| -\frac{1}{2}\partial_{xx}\phi^{n}(x_{j}) + \omega \phi^{n}(x_{j}) + \left| \phi^{n}(x_{j}) \right|^{2}\phi^{n}(x_{j}) \right| < 10^{-13},
	\end{align*}
\end{example}

\begin{figure}[!ht]
	\centering
	\includegraphics[width=0.42\textwidth]{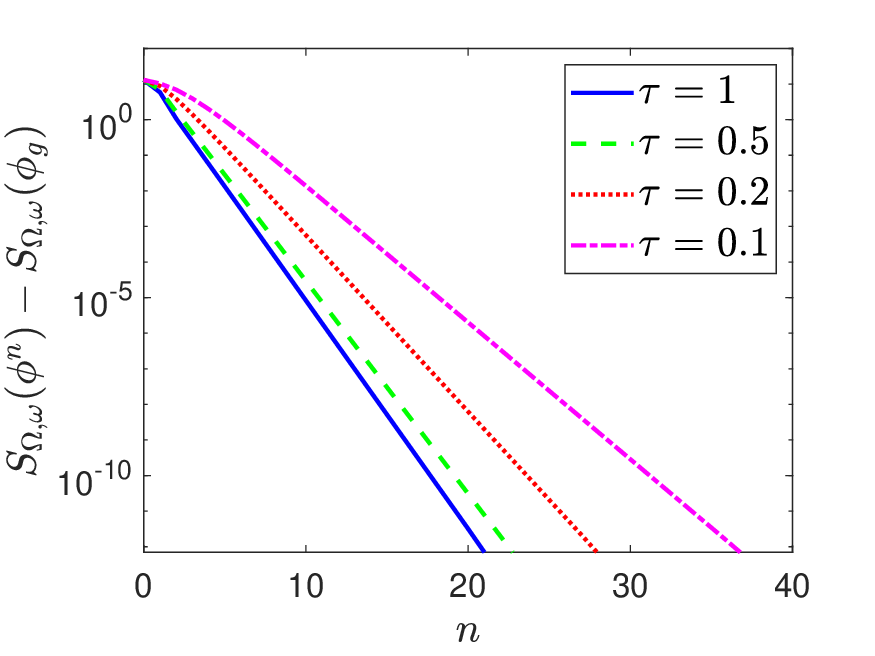} 
	\includegraphics[width=0.42\textwidth]{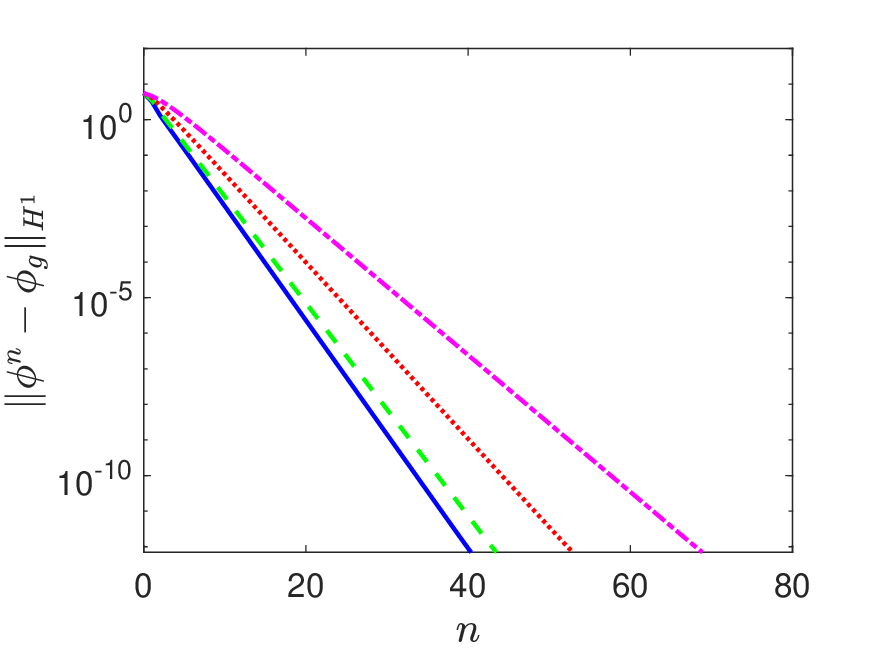}
	\caption{The change of $S_{\Omega,\omega}(\phi^{n}) - S_{\Omega,\omega}(\phi_{g})$ (left) and $H^{1}$-norm error $\|\phi^{n} - \phi_{g}\|_{H^{1}}$ (right) of DGF \eqref{eq:GF-semi-discretization} along iteration (on logarithmic scale) under different $\tau$ with $\omega = -10$ fixed in Example~\ref{ex:1}.}
	\label{fig:ex1-loglog-err}
\end{figure}

\begin{figure}[!ht]
	\centering
    \includegraphics[width=0.325\textwidth,height=0.28\textwidth]{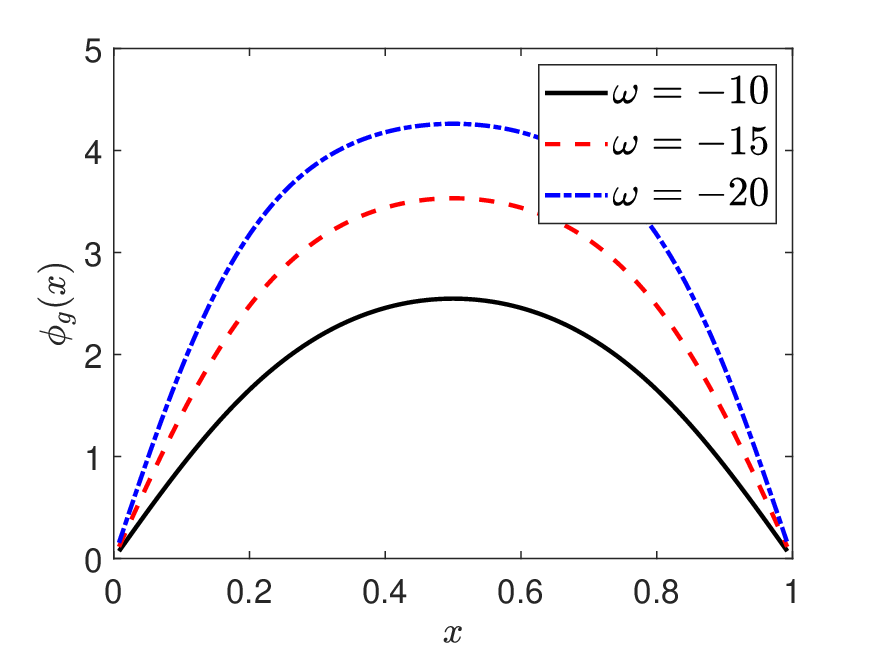}
	\includegraphics[width=0.325\textwidth,height=0.28\textwidth]{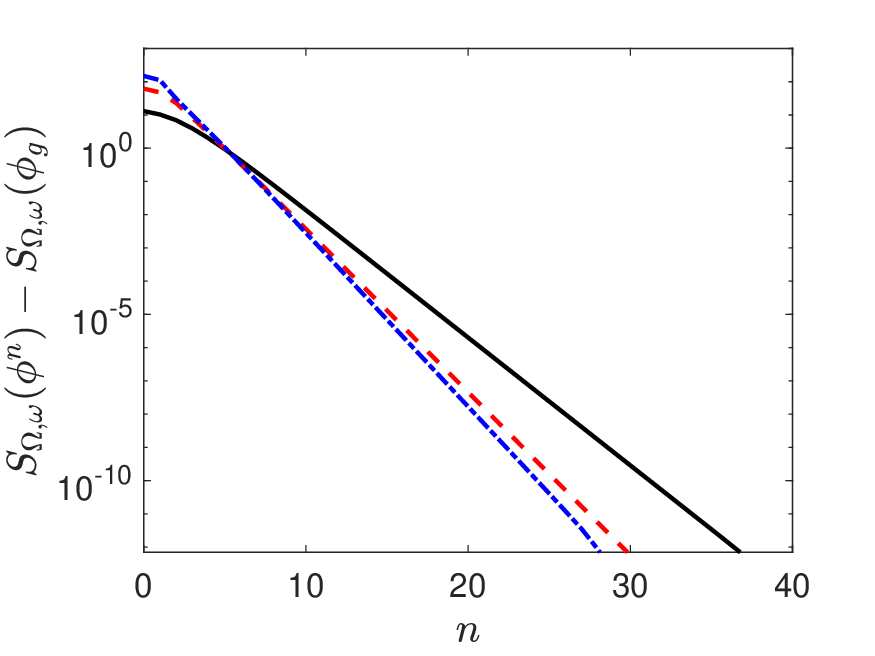} 
	\includegraphics[width=0.325\textwidth,height=0.28\textwidth]{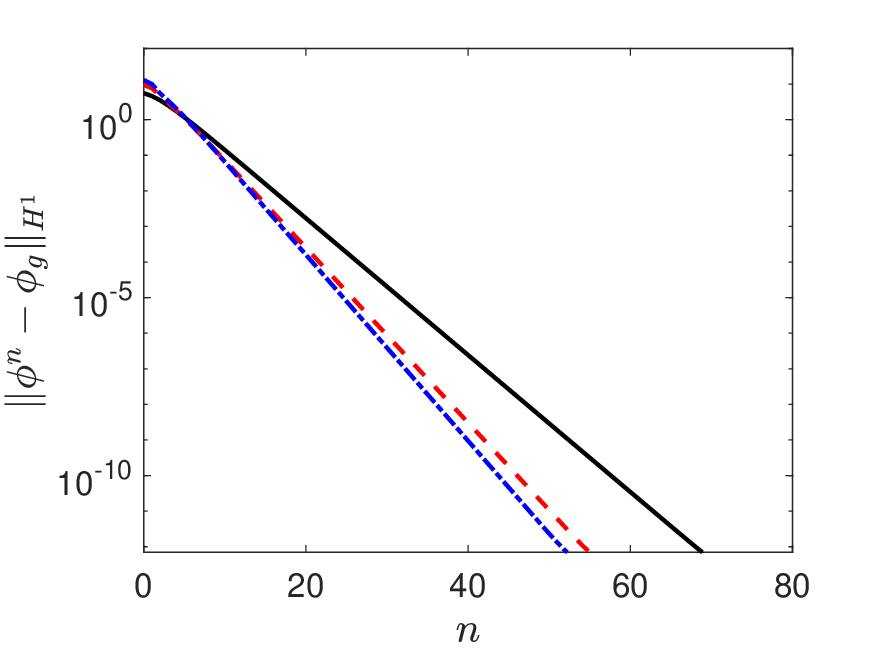}
	\caption{The profiles of GS $\phi_{g}$ (left), the change of $S_{\Omega,\omega}(\phi^{n}) - S_{\Omega,\omega}(\phi_{g})$ (middle) and $H^{1}$-norm error $\|\phi^{n} - \phi_{g}\|_{H^{1}}$ (right) of DGF \eqref{eq:GF-semi-discretization} along iteration (on logarithmic scale) under different $\omega$ with $\tau = 0.1$ fixed in Example~\ref{ex:1}.}
	\label{fig:ex1-loglog-err-omega}
\end{figure}

In Figure~\ref{fig:ex1-loglog-err-omega}, the left subplot shows the profiles of GS $\phi_{g}$ for different values of $\omega=-10,\, -15,\, -20$. Subsequently, the two plots in Figure~\ref{fig:ex1-loglog-err} together with the middle and right subplots in Figure~\ref{fig:ex1-loglog-err-omega} illustrate the evolution of the quantity $S_{\Omega,\omega}(\phi^{n}) - S_{\Omega,\omega}(\phi_{g})$ and the error $\|\phi^{n} - \phi_{g}\|_{H^{1}}$ with respect to the iteration index $n$, for different time step sizes ($\tau=1,\,0.5,\,0.2,\,0.1$) or different $\omega$ values ($\omega=-10,\, -15,\, -20$). It can be observed that all curves exhibit a distinct exponential decay in $n$.

\begin{example}[Vortex states in 2D]\label{ex:2}
	Letting $d=2$, $\mathbf{x}=(x_1,x_2)$, $p=3$, $\beta = 100$, and $V(\mathbf{x}) = \frac{1}{2}\left( \gamma_{1}^{2}x_{1}^{2} + \gamma_{2}^{2}x_{2}^{2} \right)$ with $\gamma_{1} = \gamma_{2} = 1$ in \eqref{eq:GF-semi-discretization}. The computation domain is fixed as $D = [-10, 10]^{2}$ with mesh size $h = 1/8$, and the time step is fixed as $\tau=0.1$, and use the Fourier pseudospectral discretization.
\end{example}

The stabilization factor is chosen adaptively according to \eqref{eq:alpha_n}, and stopping criterion is set as
\begin{align*}
	\left| S_{\Omega,\omega}(\phi^{n}) - S_{\Omega,\omega}(\phi^{n+1}) \right| < 10^{-12}.
\end{align*}
We introduce the functions $\phi_{a} = \sqrt{\gamma_{1}\gamma_{2} / \pi}\, \mathrm e^{-(\gamma_{1}x_{1}^{2} + \gamma_{2}x_{2}^{2})/2}$, and  define the following initial data under different $\Omega$ (as considered in \cite{liu2023computing}):
\begin{align}
	\phi^{0} = \begin{cases}
		\left( x_{1} + \mathrm{i} x_{2} \right) \psi_{a}, & \Omega = 0.3, \\
		\left( x_{1} + \mathrm{i} x_{2} \right)^{4} \psi_{a}, & \Omega = 0.5, \\
		\frac{1}{2}\left( \left( x_{1} + \mathrm{i} x_{2} \right) + \left( x_{1} + \mathrm{i} x_{2} \right)^{4} \right)\psi_{a}, & \Omega = 0.7.
	\end{cases} \label{eq:phi0-experiment}
\end{align}
 Under these initial values, the final steady state exhibits quantized vortices, as shown in Figure~\ref{fig:ex2-phig}.

\begin{figure}[!ht]
	\centering
	\includegraphics[width=0.325\textwidth,height=0.28\textwidth]{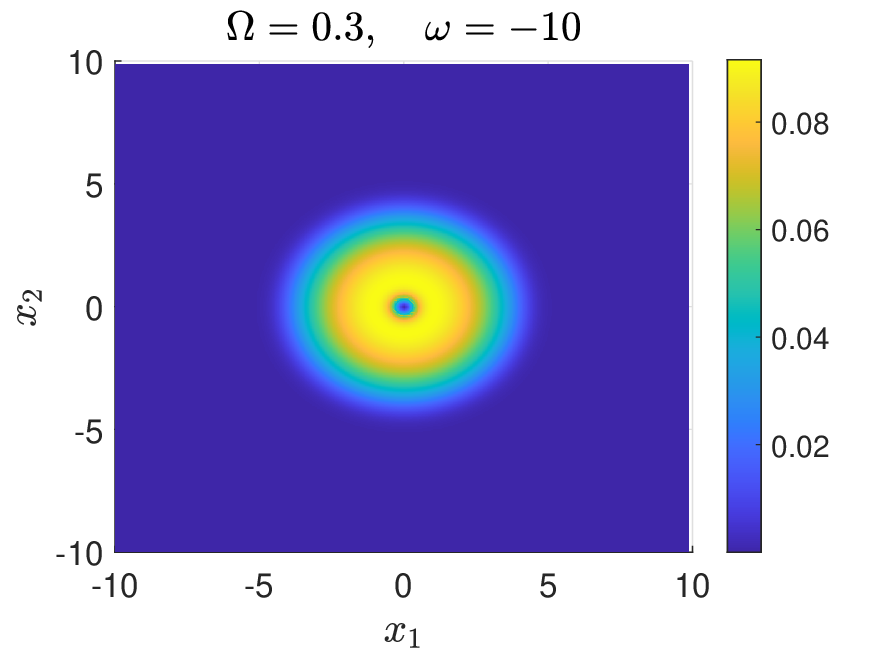} 
	\includegraphics[width=0.325\textwidth,height=0.28\textwidth]{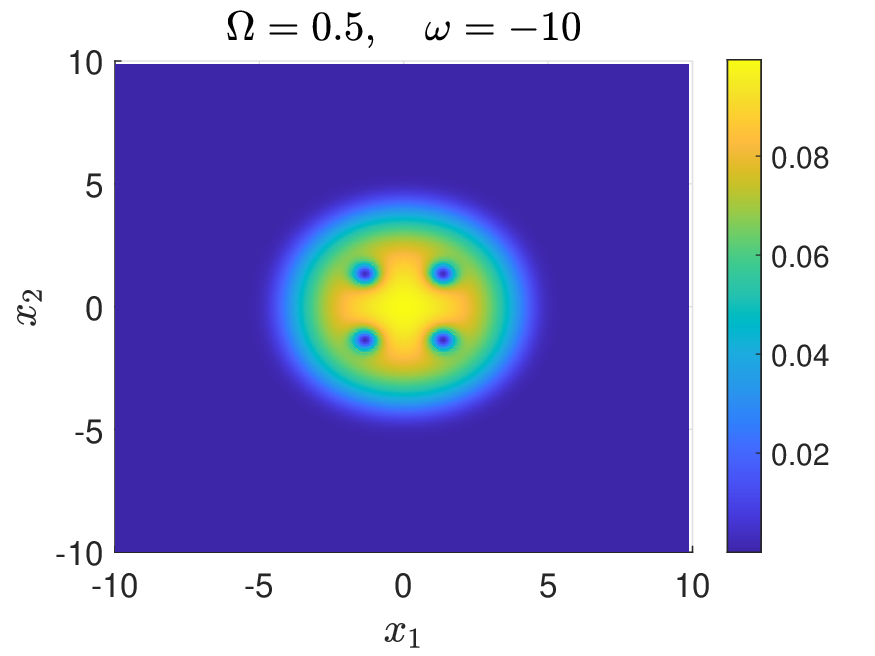}
	\includegraphics[width=0.325\textwidth,height=0.28\textwidth]{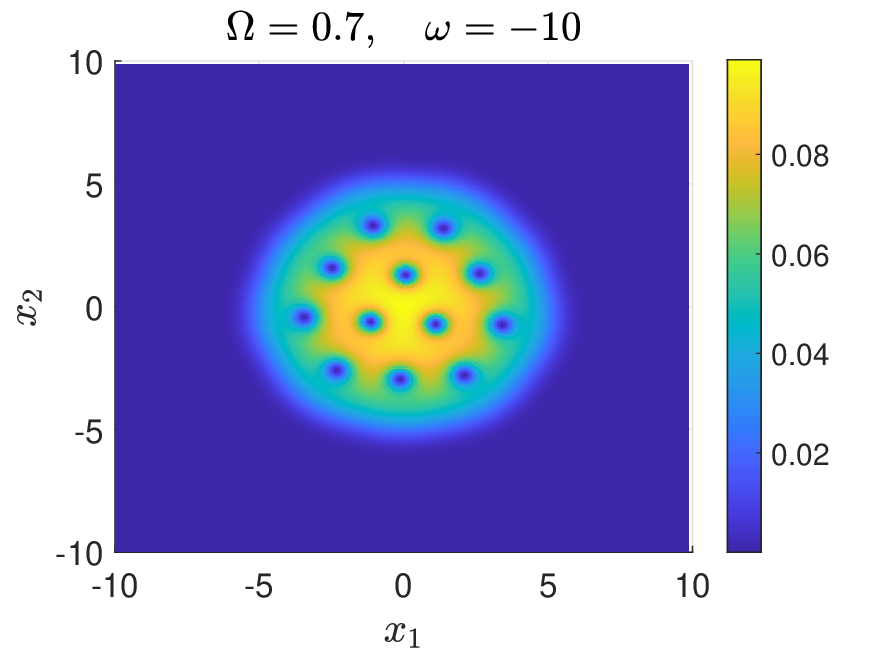} \\
    \includegraphics[width=0.325\textwidth,height=0.28\textwidth]{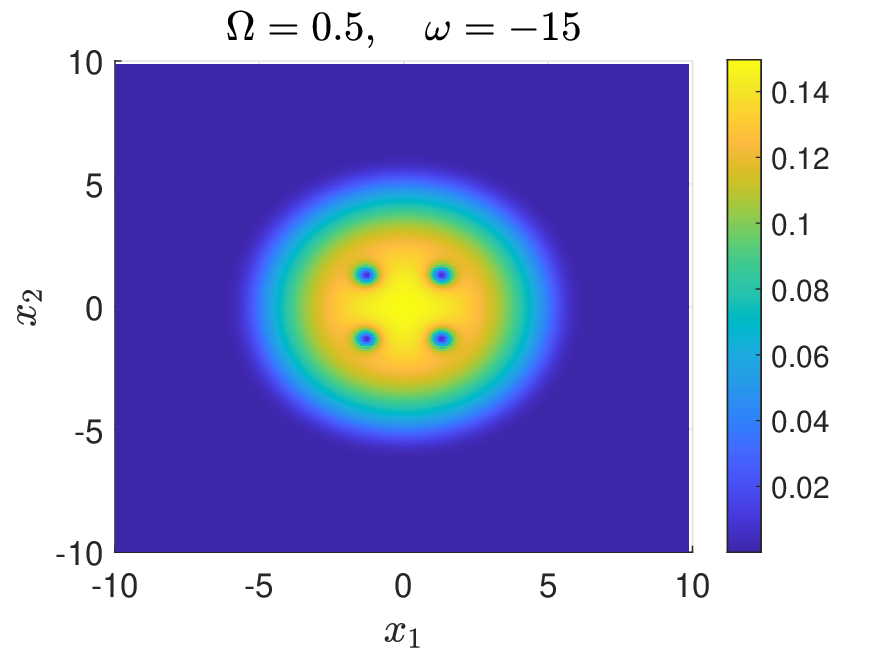}
    \includegraphics[width=0.325\textwidth,height=0.28\textwidth]{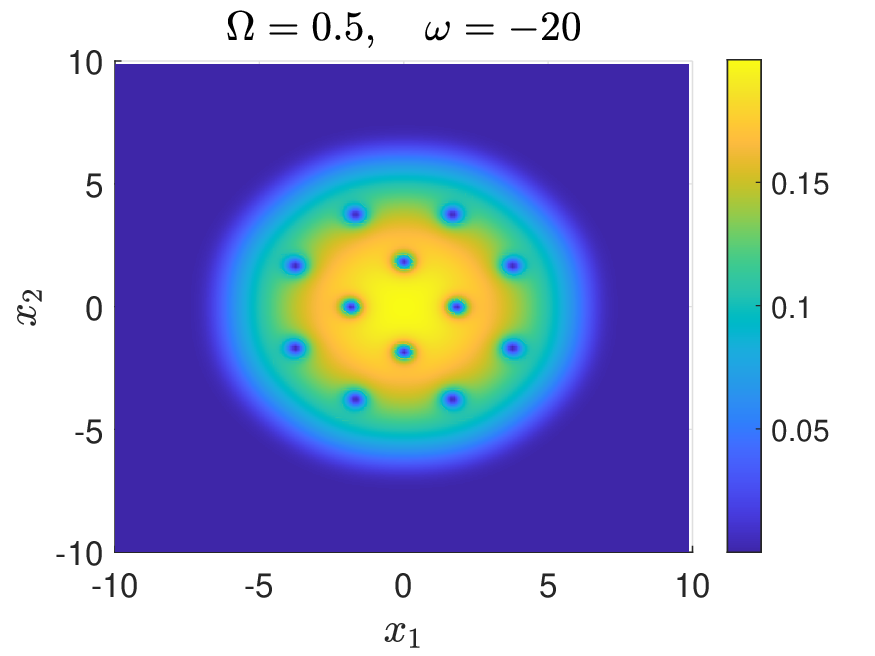}
	\caption{Contour plots of $|\phi_{g}|^{2}$ for different $\Omega$ and $\omega$ with $\phi^{0}$ of \eqref{eq:phi0-experiment} in Example~\ref{ex:2}.}
	\label{fig:ex2-phig}
\end{figure}

\begin{figure}[!ht]
	\centering
	\includegraphics[width=0.325\textwidth,height=0.28\textwidth]{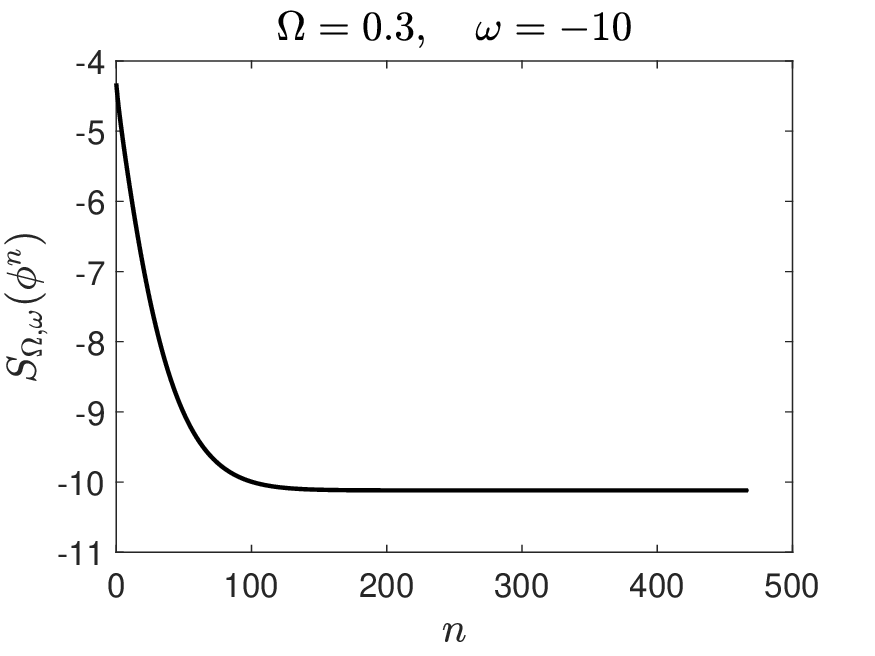} 
	\includegraphics[width=0.325\textwidth,height=0.28\textwidth]{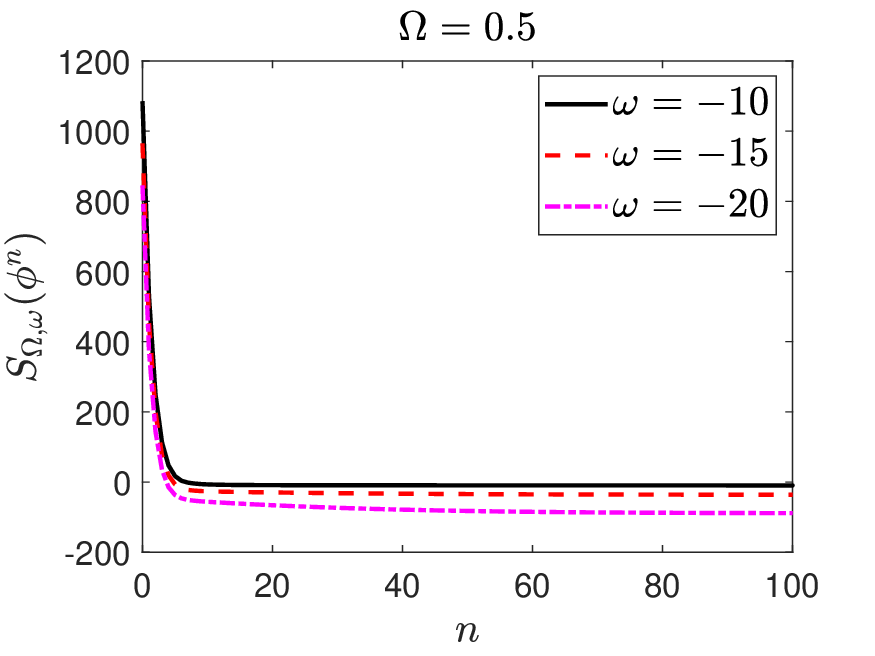}
	\includegraphics[width=0.325\textwidth,height=0.28\textwidth]{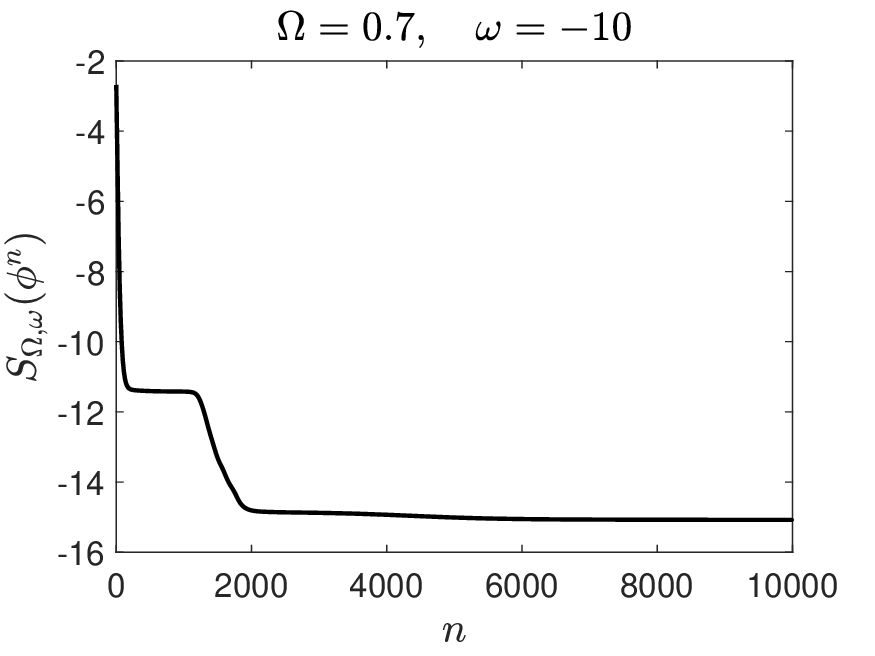} 
	\caption{$S_{\Omega,\omega}(\phi^{n})$ for different $\Omega$ and $\omega$ with $\phi^{0}$ of \eqref{eq:phi0-experiment} in Example~\ref{ex:2}.}
	\label{fig:ex2-Sphi}
\end{figure}

\begin{figure}[!ht]
	\centering
	\includegraphics[width=0.325\textwidth,height=0.28\textwidth]{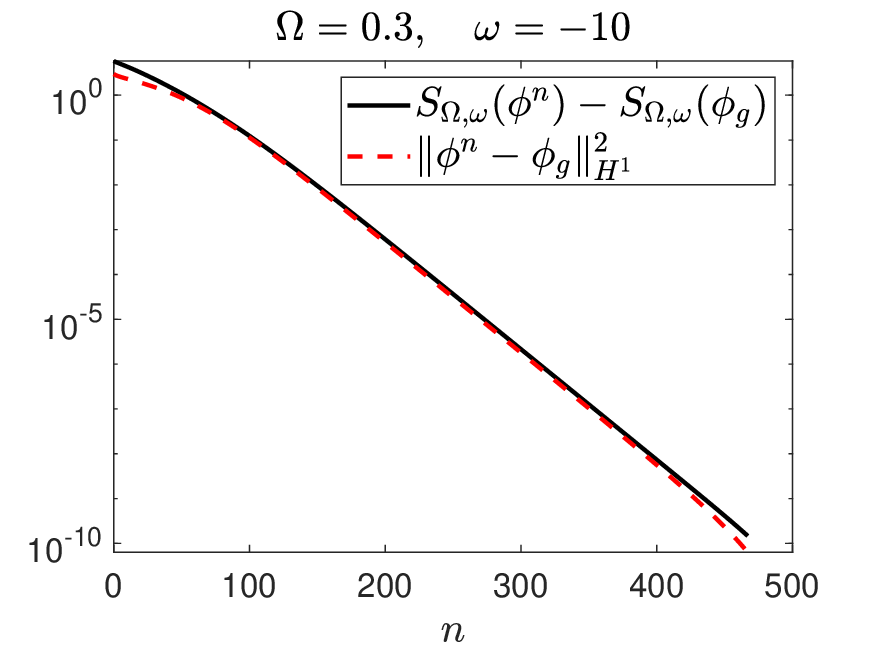} 
	\includegraphics[width=0.325\textwidth,height=0.28\textwidth]{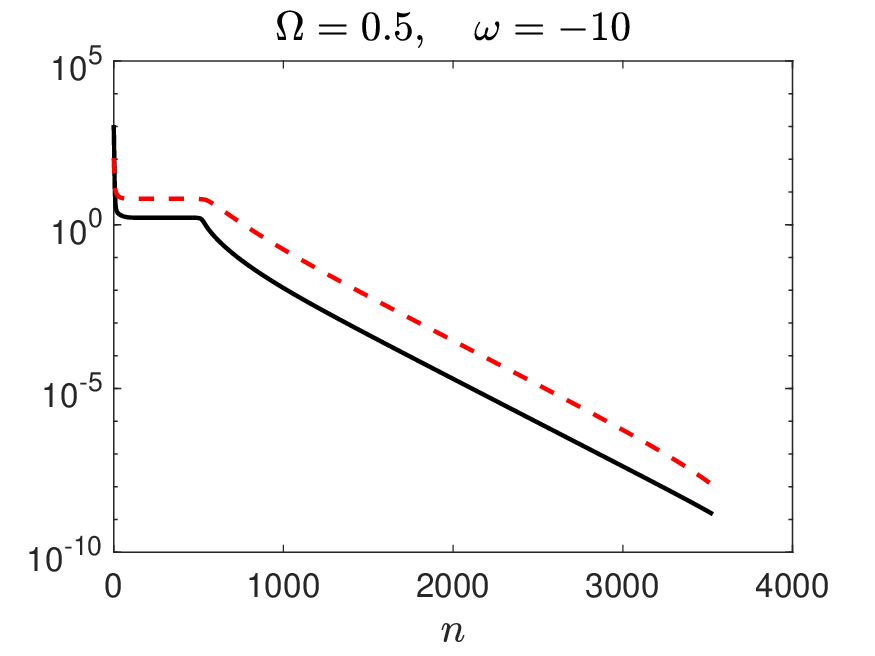}
	\includegraphics[width=0.325\textwidth,height=0.28\textwidth]{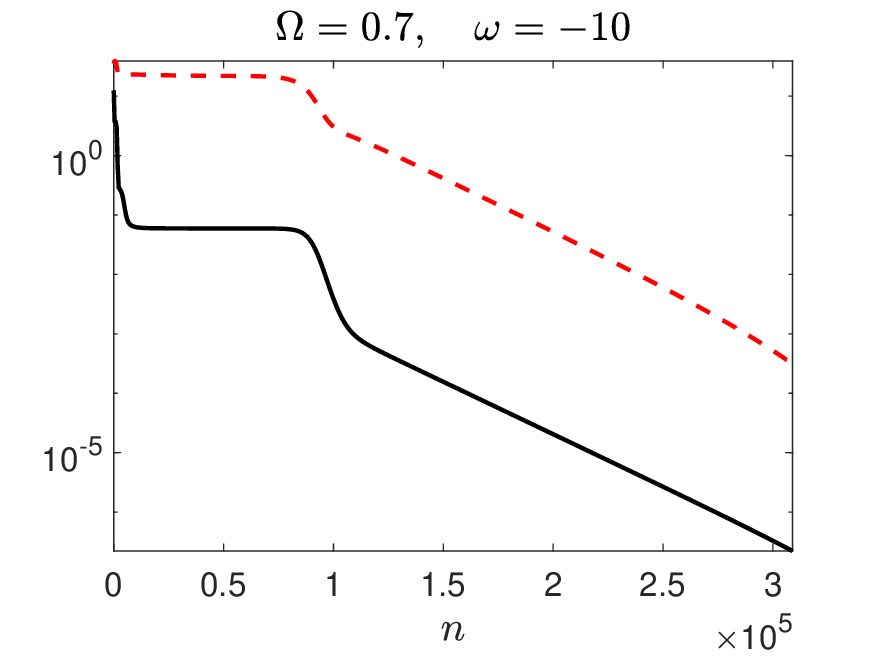} \\
    \includegraphics[width=0.325\textwidth,height=0.28\textwidth]{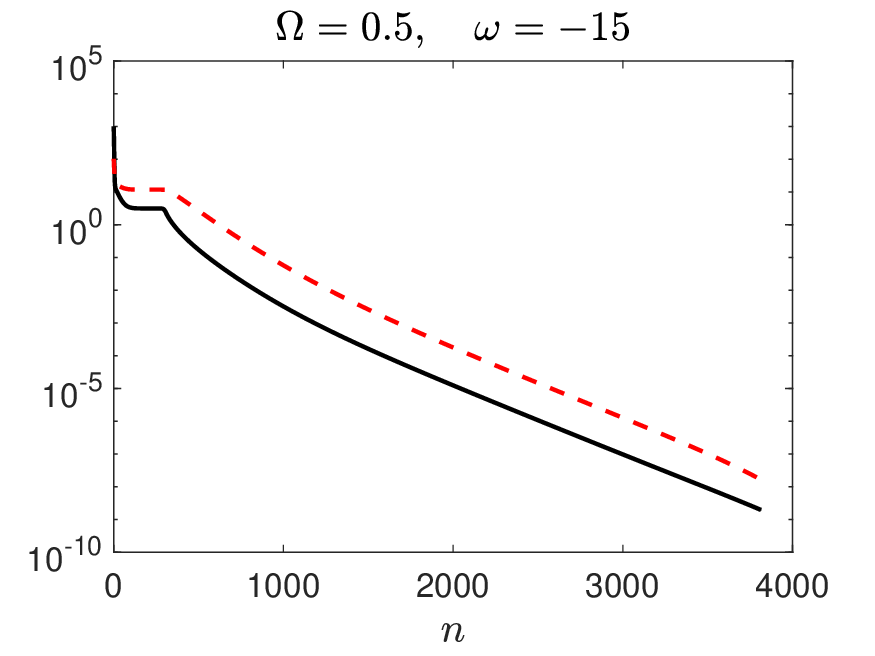}
    \includegraphics[width=0.325\textwidth,height=0.28\textwidth]{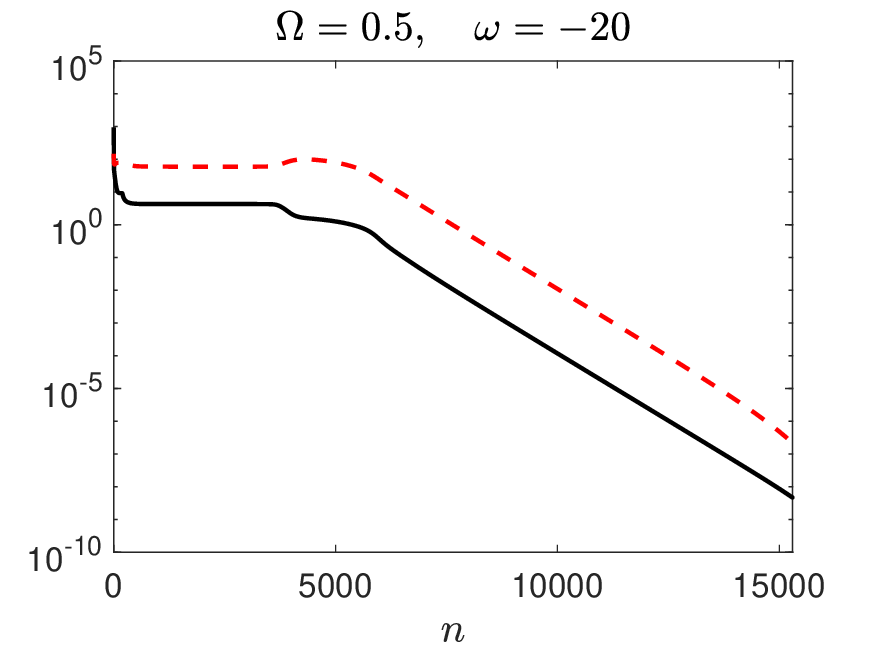}
	\caption{The change of $S_{\Omega,\omega}(\phi^{n}) - S_{\Omega,\omega}(\phi_{g})$ and  $\|\phi^{n} - \phi_{g}\|_{H^{1}}^{2}$ of DGF \eqref{eq:GF-semi-discretization} along iteration (on logarithmic scale) under different $\Omega$ and $\omega$ with $\phi^{0}$ of \eqref{eq:phi0-experiment} in Example~\ref{ex:2}.}
	\label{fig:ex2-loglog-err}
\end{figure}

Figure~\ref{fig:ex2-Sphi} illustrates the evolution of the action functional $S_{\Omega,\omega}(\phi^{n})$ along iteration under different angular velocities $\Omega$ and $\omega$ with initial conditions (specified in \eqref{eq:phi0-experiment}). A key observation is the monotonic decrease of $S_{\Omega,\omega}$ in all cases. Figure~\ref{fig:ex2-loglog-err} illustrates the decay of two key quantities: the energy difference $S_{\Omega,\omega}(\phi^{n}) - S_{\Omega,\omega}(\phi_{g})$ and the squared $H^{1}$-error $\|\phi^{n} - \phi_{g}\|_{H^{1}}^{2}$, plotted against the time step index $n$. Here, $\phi_{g}$ denotes the numerical GS obtained by the DGF method upon satisfying the stopping criterion. In all cases, both curves eventually exhibit a clear linear decay on the log-log scale, which corresponds to an exponential decay in the original variables. Notably, during this exponential decay phase, the two curves are parallel. This parallelism aligns with our theoretical result on the equivalence between the deviation of $S_{\Omega, \omega}$ and the square of the $H^{1}$ distance to the GS. Notably, for $\Omega = 0.5$ and $0.7$, the curves do not exhibit exponential decay during the initial phase of the iteration; instead, they remain relatively flat over a significant interval. This behavior aligns perfectly with our theoretical analysis of local convergence: the exponential decay rate is only guaranteed once the numerical solution enters a sufficiently small neighborhood of the GS $\phi_{g}$.

\section{Conclusion} \label{sec. 5}
This work has analyzed the direct gradient flow (DGF) method for computing defocusing action ground states of rotating nonlinear Schr\"odinger equations (RNLS). We have rigorously established the unconditional stability of the DGF scheme, demonstrating that the action functional is monotonically non-increasing along the discrete flow for any time step size. Furthermore, we have provided a comprehensive convergence analysis, proving both global convergence under mild assumptions and, crucially, local convergence with an optimal exponential rate toward the action ground state under a reasonable non-degeneracy condition. Numerical experiments have validated our theoretical findings, demonstrating that the DGF scheme exhibits a clear exponential convergence behavior after the solution enters a neighborhood of the ground state. The close agreement between the theoretical predictions and the numerical results confirms the effectiveness and robustness of the DGF method for computing defocusing action ground states. Future work will focus on extending this analysis to the fully discrete setting, incorporating spatial discretization methods such as finite element or pseudospectral approaches.

\section*{acknowledgment}
W. Liu is supported by the NSFC grant 12571448 and the Innovation Research Foundation of NUDT (202402-YJRC-XX-002). T. Wang and X. Zhao are supported by National Key Research and Development Program of China, National MCF Energy R\&D Program (No. 2024YFE03240400), NSFC 42450275, 12271413. Y. Yuan is supported by the NSFC grant 12471375.

\bibliographystyle{siam}
\bibliography{references}

\end{document}